\documentclass{amsproc}
\usepackage{amsfonts}

\theoremstyle{plain}
\newtheorem{theorem}{Theorem}
\newtheorem{lemma}{Lemma}
\newtheorem{corollary}{Corollary}
\newtheorem{proposition}{Proposition}

\theoremstyle{definition}

\theoremstyle{remark}
\newtheorem{remark}{Remark}

\numberwithin{equation}{section}
\input{tcilatex}

\begin{document}
\title[Askey--Wilson polynomials with complex parameters.]{On the structure
and probabilistic interpretation of Askey--Wilson densities and polynomials
with complex parameters.}
\author{Pawe\l\ J. Szab\l owski}
\address{Department of Mathematics and Information Sciences,\\
Warsaw University of Technology\\
pl. Politechniki 1, 00-661 Warsaw, Poland}
\email{pawel.szablowski@gmail.com}
\date{July, 2010}

\begin{abstract}
We give equivalent forms of the Askey--Wilson polynomials expressing them
with the help of the Al-Salam--Chihara polynomials. After restricting
parameters of the Askey--Wilson polynomials to complex conjugate pairs we
expand the Askey--Wilson weight function in the series similar to the
Poisson--Mehler expansion formula and give its probabilistic interpretation.
In particular this result can be used to calculate explicit forms of
'q-Hermite' moments of the Askey--Wilson density, hence enabling calculation
of all moments of the Askey--Wilson density. On the way (by setting certain
parameter $q$ to $0$) we get some formulae useful in the rapidly developing
so called 'free probability'.
\end{abstract}

\thanks{The author is grateful to unknown referee for pointing out
simplifications of some proofs and errors in paper's English.}
\maketitle

\section{Introduction}

The aim of this paper is to present some properties of the Askey--Wilson
(briefly AW) polynomials and their weight function. This is the function
that makes these polynomials orthogonal. As it is well known (see e.g. \cite%
{IA}) the AW polynomials are characterized by 5 parameters one of which is
special, traditionally denoted by $q$, often called a base. In majority of
cases $-1<q\leq 1$. The parameter $q$ plays a special r\^{o}le that will be
exposed in the sequel. The remaining 4 parameters can be either real or
complex but forming conjugate pairs. If products of all pairs of these 4
parameters have absolute values less than 1 then there exists a positive
measure on a compact segment that makes the AW polynomials orthogonal. If
absolute values of all parameters are less than 1 then this measure has a
density. If all $4$ parameters are complex and are in conjugate pairs then
the AW weight can be scaled to be the probability density having nice
probabilistic interpretation. This density for $q\allowbreak =\allowbreak 1$
is nothing else but one of the conditional densities of certain $3$
dimensional jointly Normal distribution. We will explain it in the sequel.
Because of these interpretations our main concern will be with the complex
parameter case. In particular our main result will allow expansion of the AW
density in certain series of the so called $q-$Hermite polynomials. The
expansion is analogous to the Poisson--Mehler series.

However to present briefly and clearly our results we have to refer to the $%
q $-series theory and some of its basic notions. Although the $q-$series
theory has links with combinatorics, non-commutative analysis and
probability theory it is not widely known. That is why we will recall some
notions and facts concerning it. Our considerations and calculations are
simple and in fact elementary.

Traditionally the AW polynomials (see e.g. \cite{AW85} or \cite{IA} or \cite%
{Koek}) are defined through the finite $q-$hypergeometric series. More
precisely the $n-$th AW polynomial $D_{n}$ is defined by:%
\begin{equation*}
D_{n}\left( x|a,b,c,d,q\right) =\frac{\left( ab,ac,ad\right) _{n}}{%
a^{n}(abcdq^{n-1})_{n}}~_{4}\phi _{3}(\QATOPD. \vert
{q^{-n},abcdq^{n-1},ae^{-i\theta },ae^{i\theta }}{ab,ac,ad,q}q,q),
\end{equation*}%
where $_{4}\phi _{3}$ is the $q-$hypergeometric series defined by 
\begin{equation*}
_{r}\phi _{s}\left( \QATOPD. \vert {a_{1},\ldots ,a_{r}}{b_{1},\ldots
,b_{s},q}q,y\right) =\sum_{k=0}^{\infty }\frac{\left( a_{1},\ldots
,a_{r}\right) _{k}}{\left( b_{1},\ldots ,b_{s},q\right) _{k}}%
(-1)^{s+1-r}q^{(s+1-r)\binom{k}{2}}y^{k},
\end{equation*}%
where $\binom{n}{k}$ is the binomial coefficient and $x\allowbreak
=\allowbreak \cos \theta .$ $\left( a_{1},\ldots ,a_{r}\right) _{k}$ and $%
\left( b_{1},\ldots ,b_{s},q\right) _{k}$ as well as $\left( ab,ac,ad\right)
_{n}$ and $(abcdq^{n-1})_{n}$ are explained at the beginning of the next
section.

The above mentioned form is difficult to use and analyze by those who do not
work in the special function theory. On the other hand due to pioneering
works of Bo\.{z}ejko et al. \cite{Bo} and also of Bryc et al. \cite{bryc1} , 
\cite{bms}, \cite{BryMaWe}, \cite{BryWe}, \cite{BryWes10}, the Askey--Wilson
polynomials and some of their subclasses have nice, clear and classical
probabilistic interpretation. Hence interest in this family of polynomials
has not only been among specialists in special functions or those working in
the theory orthogonal polynomials, but also among specialists in the
probability theory both non-commutative and classical. Not to mention people
working in quantum mechanics or quantum groups (see e.g. \cite{Floreanini97}%
).

By setting $q\allowbreak =\allowbreak 0$ we enter the world of rapidly
developing so called 'free probability' (see e.g. \cite{Voi}, \cite{Voi2}, 
\cite{Nicu}).

The family of probabilistic models, where the AW polynomials and densities
appear, has $5$ parameters and is very versatile. Hence it can be used in a
brief descriptions of various, complicated statistical models.

We will express the Askey--Wilson polynomials as certain combinations of
simpler polynomials (the Al-Salam--Chihara polynomials introduced in section %
\ref{pomoc}). Especially simple form of the AW polynomials will be obtained
in the special case of complex, grouped in conjugate pairs, parameters.

The paper is organized as follows. Notation and known results of the $q-$%
series theory that will be of help in further calculations are presented in
Section \ref{pomoc}. Next Section \ref{glowne} presents our main results.
The following short Section \ref{open} presents some immediate open
problems. The lengthy proofs of some of the results are collected in the
last Section \ref{dowody}.

\section{Auxiliary results\label{pomoc}}

Assume that $-1<q\leq 1.$ We will use traditional notation of the $q-$series
theory i.e. $\left[ 0\right] _{q}\allowbreak =\allowbreak 0,$ $\left[ n%
\right] _{q}\allowbreak =\allowbreak 1+q+\ldots +q^{n-1}\allowbreak ,$ $%
\left[ n\right] _{q}!\allowbreak =\allowbreak \prod_{i=1}^{n}\left[ i\right]
_{q},$ with $\left[ 0\right] _{q}!\allowbreak =1,\QATOPD[ ] {n}{k}%
_{q}\allowbreak =\allowbreak \left\{ 
\begin{array}{ccc}
\frac{\left[ n\right] _{q}!}{\left[ n-k\right] _{q}!\left[ k\right] _{q}!} & 
, & n\geq k\geq 0 \\ 
0 & , & otherwise%
\end{array}%
\right. $.

It will be also helpful to use the so called $q-$Pochhammer symbol defined
for $n\geq 1$ by$:\left( a;q\right) _{n}=\prod_{i=0}^{n-1}\left(
1-aq^{i}\right) ,$ with $\left( a;q\right) _{0}=1$ , $\left(
a_{1},a_{2},\ldots ,a_{k};q\right) _{n}\allowbreak =\allowbreak
\prod_{i=1}^{k}\left( a_{i};q\right) _{n}$.

Often $\left( a;q\right) _{n}$ as well as $\left( a_{1},a_{2},\ldots
,a_{k};q\right) _{n}$ will be abbreviated to $\left( a\right) _{n}$ and $%
\left( a_{1},a_{2},\ldots ,a_{k}\right) _{n},$ if it will not cause
misunderstanding.

In particular it is easy to notice that $\left( q\right) _{n}=\left(
1-q\right) ^{n}\left[ n\right] _{q}!$ and that\newline
$\QATOPD[ ] {n}{k}_{q}\allowbreak =$\allowbreak $\allowbreak \left\{ 
\begin{array}{ccc}
\frac{\left( q\right) _{n}}{\left( q\right) _{n-k}\left( q\right) _{k}} & ,
& n\geq k\geq 0 \\ 
0 & , & otherwise%
\end{array}%
\right. $. \newline
Let us remark that $\left[ n\right] _{1}\allowbreak =\allowbreak n,\left[ n%
\right] _{1}!\allowbreak =\allowbreak n!,$ $\QATOPD[ ] {n}{k}_{1}\allowbreak
=\allowbreak \binom{n}{k},$ $\left( a;1\right) _{n}\allowbreak =\allowbreak
\left( 1-a\right) ^{n}$ and $\left[ n\right] _{0}\allowbreak =\allowbreak
\left\{ 
\begin{array}{ccc}
1 & if & n\geq 1 \\ 
0 & if & n=0%
\end{array}%
\right. ,$ $\left[ n\right] _{0}!\allowbreak =\allowbreak 1,$ $\QATOPD[ ] {n%
}{k}_{0}\allowbreak =\allowbreak 1,$ $\left( a;0\right) _{n}\allowbreak
=\allowbreak \left\{ 
\begin{array}{ccc}
1 & if & n=0 \\ 
1-a & if & n\geq 1%
\end{array}%
\right. .$

In the sequel we will use the following two simple properties of the $q-$%
Pochhammer symbol.

\begin{lemma}
\label{nawiasy}i) For $-1<q\leq 1,a\in \mathbb{R},n\geq 0:\sum_{i=0}^{n}%
\QATOPD[ ] {n}{i}_{q}a^{i}\left( a\right) _{n-i}\allowbreak =\allowbreak 1,$

ii) For $-1<q\leq 1,a,b\in \mathbb{R},n\geq 0:$\newline
$\sum_{i=0}^{n}\left( -1\right) ^{i}q^{\binom{i}{2}}\QATOPD[ ] {n}{i}%
_{q}\left( a\right) _{i}b^{i}\left( abq^{i}\right) _{n-i}\allowbreak
=\allowbreak \left( b\right) _{n}.$
\end{lemma}

\begin{proof}
An easy proof based on the so called $q-$binomial theorem (compare Thm.
10.2.1 of \cite{Andrews1999} or Thm. 12.2.5 of \cite{IA}) is shifted to the
section \ref{dowody}.
\end{proof}

Let us define the following sets of polynomials:

The $q-$Hermite polynomials defined by: 
\begin{equation}
h_{n+1}(x|q)=2xh_{n}(x|q)-(1-q^{n})h_{n-1}(x|q),  \label{q-cont}
\end{equation}%
for $n\geq 1,$ with $h_{-1}(x|q)=0,$ $h_{0}(x|q)=1$. The polynomials $h_{n}$
are also often called continuous $q-$Hermite polynomials. However we will
more frequently use the following transformed form of polynomials $h_{n},$
namely the polynomials: 
\begin{equation*}
H_{n}\left( x|q\right) \allowbreak =\allowbreak (1-q)^{-n/2}h_{n}\left( 
\frac{x\sqrt{1-q}}{2}|q\right) .
\end{equation*}%
We will call them also $q-$Hermite. The name is justified since one can
easily show that $H_{n}\left( x|1\right) \allowbreak =\allowbreak
H_{n}\left( x\right) ,$ where $H_{n}$ denotes the $n-$th ordinary, so called
probabilistic Hermite polynomial. More precisely the polynomials $\left\{
H_{n}\right\} _{n\geq -1}$ satisfy $3-$term recurrence (\ref{_1}), below: 
\begin{equation}
H_{n+1}\left( x\right) =xH_{n}\left( x\right) -nH_{n-1}(x),  \label{_1}
\end{equation}%
with $H_{0}\left( x\right) \allowbreak =\allowbreak H_{1}\left( x\right)
\allowbreak =\allowbreak 1$. Hence they are orthogonal with respect to the
measure with the density equal to $\exp \left( -x^{2}/2\right) /\sqrt{2\pi }%
. $

The polynomials $\left\{ H_{n}\left( x|q\right) \right\} $ satisfy the
following $3-$term recurrence%
\begin{equation}
H_{n+1}\left( x|q\right) \allowbreak =\allowbreak xH_{n}\left( x|q\right) - 
\left[ n\right] _{q}H_{n-1}\left( x\right) ,  \label{He}
\end{equation}%
with $H_{-1}\left( x|q\right) \allowbreak =\allowbreak 0$, $H_{1}\left(
x|q\right) \allowbreak =\allowbreak 1$.

We shall also use the following polynomials called Al-Salam--Chihara (ASC
polynomials). As before, in the literature connected with the special
functions or orthogonal polynomials as the ASC polynomials function
polynomials defined recursively:%
\begin{equation}
Q_{n+1}\left( x|a,b,q\right) \allowbreak =\allowbreak
(2x-(a+b)q^{n})Q_{n}\left( x|a,b,q\right)
-(1-abq^{n-1})(1-q^{n})Q_{n-1}(x|a,b,q),  \label{AlSC1}
\end{equation}%
with $Q_{-1}\left( x|a,b,q\right) \allowbreak =\allowbreak 0,$ $Q_{0}\left(
x|a,b,q\right) \allowbreak =\allowbreak 1$.

We will more often use these polynomials re-scaled, with new parameters $%
\rho $ and $y$ defined by: 
\begin{equation*}
a\allowbreak =\allowbreak \frac{\sqrt{1-q}}{2}\rho (y\allowbreak
-\allowbreak i\sqrt{\frac{4}{1-q}-y^{2}}),b\allowbreak =\allowbreak \frac{%
\sqrt{1-q}}{2}\rho (y\allowbreak +\allowbreak i\sqrt{\frac{4}{1-q}-y^{2}}),
\end{equation*}%
such that $y^{2}\leq 4/(1-q),$ $\left\vert \rho \right\vert <1$. In the
formula above $i$ stands for the imaginary unit.

More precisely we will consider the polynomials 
\begin{equation}
(1-q)^{n/2}P_{n}\left( x|y,\rho ,q\right) \allowbreak =\allowbreak
Q_{n}\left( x\sqrt{1-q}/2|\frac{\sqrt{1-q}}{2}\rho (y\allowbreak
-\allowbreak i\sqrt{\frac{4}{1-q}-y^{2}}),\frac{\sqrt{1-q}}{2}\rho
(y\allowbreak +\allowbreak i\sqrt{\frac{4}{1-q}-y^{2}}),q\right) .
\label{podstawienie}
\end{equation}%
One shows that polynomials $\left\{ P_{n}\right\} $ satisfy the following $3-
$term recurrence:

\begin{equation}
P_{n+1}(x|y,\rho ,q)=(x-\rho yq^{n})P_{n}(x|y,\rho ,q)-(1-\rho
^{2}q^{n-1})[n]_{q}P_{n-1}(x|y,\rho ,q),  \label{AlSC}
\end{equation}%
with $P_{-1}\left( x|y,\rho ,q\right) \allowbreak =\allowbreak 0,$ $%
P_{0}\left( x|y,\rho ,q\right) \allowbreak =\allowbreak 1$. The polynomials $%
\left\{ P_{n}\right\} $ have nice probabilistic interpretation see e.g. \cite%
{bms}. To support intuition let us remark that 
\begin{equation*}
P_{n}\left( x|y,\rho ,1\right) =(1-\rho ^{2})^{n/2}H_{n}\left( \frac{x-\rho y%
}{\sqrt{1-\rho ^{2}}}\right) .
\end{equation*}

The polynomials (\ref{He}) satisfy the following very useful identity
originally formulated for the continuous $q-$Hermite polynomials $h_{n}$
(can be found in e.g. \cite{IA} Thm. 13.1.5) and here, below presented for
the polynomials $H_{n}$:

\begin{equation}
H_{n}\left( x|q\right) H_{m}\left( x|q\right) =\sum_{j=0}^{\min \left(
n,m\right) }\QATOPD[ ] {m}{j}_{q}\QATOPD[ ] {n}{j}_{q}\left[ j\right]
_{q}!H_{n+m-2k}\left( x|q\right) .  \label{identity}
\end{equation}%
Let us denote for simplicity the following real subsets:%
\begin{equation}
S\left( q\right) =\left\{ 
\begin{array}{ccc}
\lbrack -2/\sqrt{1-q},2/\sqrt{1-q}] & if & \left\vert q\right\vert <1 \\ 
\mathbb{R} & if & q=1%
\end{array}%
\right. ,  \label{S(q)}
\end{equation}%
and the following family of quadratic, auxiliary, polynomials: 
\begin{equation}
w_{k}\left( x,y|\rho ,q\right) =(1-\rho ^{2}q^{2k})^{2}-(1-q)\rho
q^{k}(1+\rho ^{2}q^{2k})xy+(1-q)\rho ^{2}(x^{2}+y^{2})q^{2k},  \label{w_k}
\end{equation}%
$k\allowbreak =\allowbreak 0,1,2,\ldots $ . Notice that $\forall k\geq 0:$ $%
w_{k}\left( x,y|\rho ,q\right) \allowbreak =\allowbreak w_{0}\left( x,y|\rho
q^{k},q\right) $ and that $w_{k}\left( x,y|0,q\right) \allowbreak
=\allowbreak 1$

It is known (see e.g. \cite{bryc1}, but also \cite{IA} Thm. 13.1.3 with an
obvious modification for the polynomials $H_{n}$) that the $q-$Hermite
polynomials are monic and orthogonal with respect to the measure that has
the density given by:%
\begin{equation}
f_{N}\left( x|q\right) =\frac{\sqrt{1-q}\left( q\right) _{\infty }}{2\pi 
\sqrt{4-(1-q)x^{2}}}\prod_{k=0}^{\infty }\left(
(1+q^{k})^{2}-(1-q)x^{2}q^{k}\right) I_{S\left( q\right) }\left( x\right) ,
\label{qN}
\end{equation}%
defined for $\left\vert q\right\vert <1,$ $x\in \mathbb{R},$ where 
\begin{equation*}
I_{A}\left( x\right) \allowbreak =\allowbreak \left\{ 
\begin{array}{ccc}
1 & if & x\in A \\ 
0 & if & x\notin A%
\end{array}%
\right. .
\end{equation*}%
We will set also 
\begin{equation}
f_{N}\left( x|1\right) \allowbreak =\allowbreak \frac{1}{\sqrt{2\pi }}\exp
\left( -x^{2}/2\right) .  \label{q=1}
\end{equation}

Similarly it is known (e.g. from \cite{bms} and also \cite{IA} formula
15.1.5 after re-scaling polynomials $Q_{n}$ to $P_{n}$) that the polynomials 
$\left\{ P_{n}\left( x|y,\rho ,q\right) \right\} _{n\geq -1}$ are monic and
orthogonal with respect to the measure that for $q\in (-1,1]$ and $%
\left\vert \rho \right\vert <1$ has the density. For $\left\vert
q\right\vert <1$ this density is given by: 
\begin{subequations}
\label{fCN}
\begin{equation}
f_{CN}\left( x|y,\rho ,q\right) =f_{N}\left( x|q\right) \prod_{k=0}^{\infty }%
\frac{(1-\rho ^{2}q^{k})}{w_{k}\left( x,y|\rho ,q\right) }I_{S\left(
q\right) }\left( x\right) ,  \label{2}
\end{equation}%
for $x\allowbreak \in \allowbreak \mathbb{R},$ $y\allowbreak \in \allowbreak
S\left( q\right) $ and for $q\allowbreak =\allowbreak 1$ is given by: 
\end{subequations}
\begin{equation*}
f_{CN}\left( x|y,\rho ,1\right) \allowbreak =\allowbreak \frac{1}{\sqrt{2\pi
\left( 1-\rho ^{2}\right) }}\exp \left( -\frac{\left( x-\rho y\right) ^{2}}{%
2\left( 1-\rho ^{2}\right) }\right) ,
\end{equation*}%
with $x,y\in \mathbb{R}.$

It is known (see e.g. \cite{IA} formula 13.1.10) that for $\left\vert
q\right\vert <1:$%
\begin{equation}
\sup_{x\in S\left( q\right) }\left\vert H_{n}\left( x|q\right) \right\vert
\leq s_{n}\left( q\right) \left( 1-q\right) ^{-n/2},  \label{ogr_H}
\end{equation}%
where 
\begin{equation}
s_{n}\left( q\right) \allowbreak =\allowbreak \sum_{i=0}^{n}\QATOPD[ ] {n}{i}%
_{q}.  \label{Wn}
\end{equation}

We will be studying the following density 
\begin{equation}
\phi \left( x|y,\rho _{1},z,\rho _{2},q\right) \allowbreak =\allowbreak
f_{N}\left( x|q\right) \frac{\left( \rho _{1}^{2},\rho _{2}^{2}\right)
_{\infty }}{\left( \rho _{1}^{2}\rho _{2}^{2}\right) _{\infty }}%
\prod_{k=0}^{\infty }\frac{w_{k}\left( y,z|\rho _{1}\rho _{2},q\right) }{%
w_{k}\left( x,y|\rho _{1},q\right) w_{k}\left( x,z|\rho _{2},q\right) },
\label{_x|yz}
\end{equation}%
where the polynomials $w_{k}\left( s,t|\rho ,q\right) $ are defined by (\ref%
{w_k}).

For $q\allowbreak =\allowbreak 1$ we set 
\begin{equation}
\phi \left( x|y,\rho _{1},z,\rho _{2},1\right) \allowbreak =\allowbreak 
\frac{1}{\sqrt{2\pi \frac{\left( 1-\rho _{1}^{2}\right) \left( 1-\rho
_{2}^{2}\right) }{1-\rho _{1}^{2}\rho _{2}^{2}}}}\exp \left( -\frac{\left( x-%
\frac{y\rho _{1}\left( 1-\rho _{2}^{2}\right) +z\rho _{2}\left( 1-\rho
_{1}^{2}\right) }{1-\rho _{1}^{2}\rho _{2}^{2}}\right) ^{2}}{2\frac{\left(
1-\rho _{1}^{2}\right) \left( 1-\rho _{2}^{2}\right) }{1-\rho _{1}^{2}\rho
_{2}^{2}}}\right) ,  \label{_N}
\end{equation}%
that is $\phi \left( x|y,\rho _{1},z,\rho _{2},1\right) $ is the density of
the Normal distribution \newline
$N\left( \frac{y\rho _{1}\left( 1-\rho _{2}^{2}\right) +z\rho _{2}\left(
1-\rho _{1}^{2}\right) }{1-\rho _{1}^{2}\rho _{2}^{2}},\frac{\left( 1-\rho
_{1}^{2}\right) \left( 1-\rho _{2}^{2}\right) }{1-\rho _{1}^{2}\rho _{2}^{2}}%
\right) .$

We have the following important but easy Remark.

\begin{remark}
\label{U1}i) $\phi \left( x|y,\rho _{1},z,\rho _{2},q\right) \allowbreak
=\allowbreak \frac{f_{CN}\left( z|x,\rho _{2},q\right) f_{CN}\left( x|y,\rho
_{1},q\right) f_{N}\left( y|q\right) }{f_{CN}\left( z|y,\rho _{1}\rho
_{2},q\right) f_{N}\left( y|q\right) }$, hence in particular $\phi \left(
x|y,0,z,\rho _{2},q\right) \allowbreak =\allowbreak f_{CN}\left( x|z,\rho
_{2},q\right) .$

ii) $\phi \left( x|y,\rho _{1},z,\rho _{2},q\right) \allowbreak =\allowbreak
\psi (\frac{\sqrt{1-q}}{2}x|a,b,c,d,q)$ where 
\begin{eqnarray}
a\allowbreak &=&\allowbreak \frac{\sqrt{1-q}}{2}\rho _{1}(y\allowbreak
-\allowbreak i\sqrt{\frac{4}{1-q}-y^{2}}),  \label{par1} \\
b\allowbreak &=&\allowbreak \frac{\sqrt{1-q}}{2}\rho _{1}(y\allowbreak
+\allowbreak i\sqrt{\frac{4}{1-q}-y^{2}}),  \label{par2} \\
c\mathbb{\allowbreak } &\mathbb{=}&\frac{\sqrt{1-q}}{2}\rho
_{2}(z\allowbreak -\allowbreak i\sqrt{\frac{4}{1-q}-z^{2}}),  \label{par3} \\
d\allowbreak &=&\allowbreak \frac{\sqrt{1-q}}{2}\rho _{2}(z\allowbreak
+\allowbreak i\sqrt{\frac{4}{1-q}-z^{2}}).  \label{par4}
\end{eqnarray}%
and $\psi (t|a,b,c,d,q)$ is a normalized (that is multiplied by a constant
so that its integral is $1$) weight function of the AW polynomials. Compare
e.g. \cite{AW85} or \cite{IA}. Again in the formulae (\ref{par1}),...,(\ref%
{par4}) $i$ denotes the imaginary unit.
\end{remark}

From assertion i) of the Remark above it follows that the properties of the
density $\phi $ are closely related to the properties of the densities $%
f_{CN}$ and $f_{N}.$ Hence now we will recall properties of these densities
and related to them families of the polynomials $\left\{ H_{n}\left(
x|q\right) \right\} _{n\geq -1}$ and $\left\{ P_{n}\left( x|y,\rho ,q\right)
\right\} _{n\geq -1}$ that are crucial for the main results of this paper.
We will collect them in the following two Propositions:

\begin{proposition}
\label{znane}i) For $n,m\geq 0:$%
\begin{equation*}
\int_{S\left( q\right) }H_{n}\left( x|q\right) H_{m}\left( x|q\right)
f_{N}\left( x|q\right) dx\allowbreak =\allowbreak \left\{ 
\begin{array}{ccc}
0 & when & n\neq m \\ 
\left[ n\right] _{q}! & when & n=m%
\end{array}%
\right. .
\end{equation*}

ii) For $n\geq 0:$%
\begin{equation*}
\int_{S\left( q\right) }H_{n}\left( x|q\right) f_{CN}\left( x|y,\rho
,q\right) dx=\rho ^{n}H_{n}\left( y|q\right) .
\end{equation*}

iii) For $n,m\geq 0:$%
\begin{equation*}
\int_{S\left( q\right) }P_{n}\left( x|y,\rho ,q\right) P_{m}\left( x|y,\rho
,q\right) f_{CN}\left( x|y,\rho ,q\right) dx\allowbreak =\allowbreak \left\{ 
\begin{array}{ccc}
0 & when & n\neq m \\ 
\left( \rho ^{2}\right) _{n}\left[ n\right] _{q}! & when & n=m%
\end{array}%
\right. .
\end{equation*}

iv) 
\begin{equation*}
\int_{S\left( q\right) }f_{CN}\left( x|y,\rho _{1},q\right) f_{CN}\left(
y|z,\rho _{2},q\right) dy=f_{CN}\left( x|z,\rho _{1}\rho _{2},q\right) .
\end{equation*}

v) For $|t|,\left\vert q\right\vert <1:$%
\begin{equation*}
\sum_{i=0}^{\infty }\frac{s_{i}\left( q\right) t^{i}}{\left( q\right) _{i}}%
\allowbreak =\allowbreak \frac{1}{\left( t\right) _{\infty }^{2}}%
,\sum_{i=0}^{\infty }\frac{s_{i}^{2}\left( q\right) t^{i}}{\left( q\right)
_{i}}\allowbreak =\allowbreak \frac{\left( t^{2}\right) _{\infty }}{\left(
t\right) _{\infty }^{4}},
\end{equation*}%
convergence is absolute, where $s_{i}\left( q\right) $ is defined by (\ref%
{Wn}).

vi) For $(1-q)\max (x^{2},y^{2})\leq 4,$ $\left\vert \rho \right\vert <1:$%
\begin{equation}
f_{CN}\left( x|y,\rho ,q\right) \allowbreak \allowbreak =\allowbreak
f_{N}\left( x|q\right) \sum_{n=0}^{\infty }\frac{\rho ^{n}}{[n]_{q}!}%
H_{n}(x|q)H_{n}(y|q),  \label{P-M}
\end{equation}
convergence is absolute in $\rho ,$ $y$ \& $x$ and uniform in $x$ and $y$.

vii) $\forall x,y\in S\left( q\right) :0<C\left( y,\rho ,q\right) \leq \frac{%
f_{CN}\left( x|y,\rho ,q\right) }{f_{N}\left( x|q\right) }\leq \frac{\left(
\rho ^{2}\right) _{\infty }}{\left( \rho \right) _{\infty }^{4}}.$
\end{proposition}

\begin{proof}
i) It is formula 13.1.11 of \cite{IA} with an obvious modification for the
polynomials $H_{n}$ instead of $h_{n}$ (compare (\ref{q-cont})) and the
normalized weight function (i.e. $f_{N})$ ii) Exercise 15.7 of \cite{IA}
also in \cite{bryc1}, iii) Formula 15.1.5 of \cite{IA} with an obvious
modification for the polynomials $P_{n}$ instead of $Q_{n}$ and the
normalized weight function (i.e. $f_{CN}),$ iv) see (2.6) of \cite{bms}. v)
Exercise 12.2(b) and 12.2(c) of \cite{IA}. vi) It is the famous
Poisson--Mehler formula (see e.g. \cite{IA}, for the simple proof of it see 
\cite{Szab4}).

vii) The upper limit follows directly (\ref{P-M}) and assertion v). To get
the lower one let us notice that from (\ref{fCN}) we have: $\frac{%
f_{CN}\left( x|y,\rho ,q\right) }{f_{N}\left( x|q\right) }\allowbreak =$%
\newline
$\allowbreak \prod_{k=0}^{\infty }\frac{1-\rho ^{2}q^{k}}{w_{k}\left(
x,y|\rho ,q\right) }$. Now let us notice also that $w_{k}\left( x,y|\rho
,q\right) \allowbreak =\allowbreak (1-q)\rho ^{2}q^{2k}(x\allowbreak
-\allowbreak (\rho ^{-1}q^{-k}\allowbreak +\allowbreak \rho
q^{k})y/2)^{2}\allowbreak +\allowbreak (1-\rho
^{2}q^{2k})^{2}(1-(1-q)y^{2}/4)\geq 0$. As a nonnegative quadratic form this
expression assumes its maximum value for $x\in S\left( q\right) $ at the
ends of this interval, so $(1-\rho ^{2}q^{2k})^{2}\allowbreak -\allowbreak
(1-q)\rho q^{k}(1+\rho ^{2}q^{2k})xy\allowbreak +\allowbreak (1-q)\rho
^{2}(x^{2}+y^{2})q^{2k}\allowbreak \leq $\allowbreak\ $(1-\rho
^{2}q^{2k})^{2}\allowbreak +\allowbreak 2\left( 1-q\right) (1+\rho
^{2}q^{2k})|y\rho q^{k}|\allowbreak +\allowbreak 4\rho ^{2}q^{2k}\allowbreak
+\allowbreak (1-q)\rho ^{2}y^{2}q^{2k}\allowbreak =\allowbreak (1+\rho
^{2}q^{2k})^{2}\allowbreak +\allowbreak 2\left( 1-q\right) (1+\rho
^{2}q^{2k})|y\rho q^{k}|\allowbreak +\allowbreak (1-q)\rho ^{2}y^{2}q^{2k}$.
Hence $\frac{f_{CN}\left( x|y,\rho ,q\right) }{f_{N}\left( x|q\right) }%
\allowbreak \geq \allowbreak \frac{\left( \rho ^{2}\right) _{\infty }}{%
\prod_{k=0}^{\infty }(1+\rho ^{2}q^{2k})^{2}\allowbreak +\allowbreak 2\left(
1-q\right) (1+\rho ^{2}q^{2k})|y\rho q^{k}|\allowbreak +\allowbreak
(1-q)\rho ^{2}y^{2}q^{2k}}\allowbreak \overset{df}{=}\allowbreak C\left(
y,\rho ,q\right) $.
\end{proof}

\begin{remark}
From the assertion v) of the Lemma above it follows that $\phi \left(
x|y,\rho _{1},z,\rho _{2},q\right) $ is the conditional density of $X|Y,Z$
if the joint density of $(Y,X,Z)$ is equal to \newline
$f_{N}\left( y|q\right) f_{CN}\left( x|y,\rho _{1},q\right) f_{CN}\left(
z|x,\rho _{2},q\right) .$ It is so since then the marginal density of $%
\left( Y,Z\right) $ is equal to $f_{N}\left( y|q\right) f_{CN}\left(
z|y,\rho _{1}\rho _{2},q\right) $ (which follows directly from assertion iv)
of Proposition \ref{znane}.
\end{remark}

Properties of the polynomial sets $\left\{ H_{n}\left( x|q\right) \right\}
_{n\geq -1}$ and $\left\{ P_{n}\left( x|y,\rho ,q\right) \right\} _{n\geq
-1} $ are collected in the second proposition.

We will use also the following, auxiliary set of polynomials $\left\{
B_{n}\left( x|q\right) \right\} _{n\geq -1}$ defined by:%
\begin{equation}
B_{n+1}\left( y|q\right) \allowbreak =\allowbreak -q^{n}yB_{n}\left(
y|q\right) +q^{n-1}\left[ n\right] _{q}B_{n-1}\left( y|q\right) ;n\geq 0,
\label{_B}
\end{equation}%
with $B_{-1}\left( y|q\right) =0,$ $B_{0}\left( y|q\right) =1$. The
polynomials $\left\{ B_{n}\right\} _{n\geq -1}$ with this normalization were
introduced and some of their basic properties were exposed in \cite{bms}.
However they were known earlier with different scaling and normalization
(see e.g. \cite{ASK89} or \cite{IsMas94} where polynomials $h_{n}\left(
y|q^{-1}\right) $ are analyzed). In particular it was shown in \cite{bms}
that $B_{n}\left( x|1\right) \allowbreak =\allowbreak i^{n}H_{n}\left(
ix\right) $. We will need also these polynomials with a another scaling and
normalization and also some additional properties of them. Namely we will
need 'continuous version' of these polynomials: 
\begin{equation*}
b_{n}\left( y|q\right) \allowbreak =\allowbreak (1-q)^{n/2}B_{n}\left( 2y/%
\sqrt{1-q}|q\right) .
\end{equation*}%
It is elementary to notice that the polynomials $b_{n}$ satisfy $3-$term
recurrence:%
\begin{equation}
b_{n+1}\left( y|q\right) =-2q^{n}yb_{n}\left( y|q\right)
+q^{n-1}(1-q^{n})b_{n-1}\left( y|q\right) ,  \label{maleB}
\end{equation}%
with $b_{-1}\left( y|q\right) \allowbreak =\allowbreak 0,$ $b_{0}\left(
y|q\right) \allowbreak =\allowbreak 1$. Further let us notice (comparing (%
\ref{maleB}) and (\ref{q-cont})) that 
\begin{equation}
(-1)^{n}q^{-\binom{n}{2}}b_{n}\left( y|q\right) \allowbreak =\allowbreak
h_{n}\left( y|q^{-1}\right) .  \label{h(q^-1)}
\end{equation}

\begin{proposition}
\label{connection}i) $\forall n\geq 1:P_{n}\left( x|y,\rho ,q\right)
=\sum_{j=0}^{n}\QATOPD[ ] {n}{j}\rho ^{n-j}B_{n-j}\left( y|q\right)
H_{j}\left( x|q\right) ,$

ii) $\forall n>0:\sum_{j=0}^{n}\QATOPD[ ] {n}{j}B_{n-j}\left( x|q\right)
H_{j}\left( x|q\right) =0,$

iii) $\forall n\geq 0:H_{n}\left( x|q\right) =\sum_{j=0}^{n}\QATOPD[ ] {n}{j}%
\rho ^{n-j}H_{n-j}\left( y|q\right) P_{j}\left( x|y,\rho ,q\right) .$
\end{proposition}

\begin{proof}
i) and ii) are proved in \cite{bms}. iii) Follows after inserting $P_{j}$
given by i), changing the order of summation and applying ii). However iii)
was known earlier, was given by formula (4.7) in \cite{IRS99} for
polynomials $h_{n}$ and $Q_{n}(x|a,b,q).$
\end{proof}

We will also need the following additional properties of polynomials $%
\left\{ H_{n}\left( x|q\right) \right\} $ and $\left\{ B_{n}\left(
x|q\right) \right\} $.

\begin{lemma}
\label{BH}i) 
\begin{equation*}
B_{n}\left( x|q\right) \allowbreak =\allowbreak (-1)^{n}q^{\binom{n}{2}%
}\sum_{k=0}^{\left\lfloor n/2\right\rfloor }\QATOPD[ ] {n}{k}_{q}\QATOPD[ ] {%
n-k}{k}_{q}\left[ k\right] _{q}!q^{-k(n-k)}H_{n-2k}\left( x|q\right) ,
\end{equation*}

Let us denote $I_{n,m}(x|q)\allowbreak =\allowbreak \sum_{i=0}^{n}\QATOPD[ ]
{n}{i}_{q}B_{n-i}\left( x|q\right) H_{i+m}\left( x|q\right) ,$ then

ii) 
\begin{equation*}
I_{n,m}\left( x|q\right) \allowbreak =\allowbreak -\sum_{k=1}^{n}\QATOPD[ ] {%
m}{k}_{q}\QATOPD[ ] {n}{k}_{q}\left[ k\right] _{q}!I_{n-k,m-k}\left(
x|q\right) ,
\end{equation*}

iii) 
\begin{equation*}
I_{n,m}(x|q)\allowbreak =\allowbreak \left\{ 
\begin{array}{ccc}
0 & if & n>m \\ 
(-1)^{n}q^{\binom{n}{2}}\frac{\left[ m\right] _{q}!}{\left[ m-n\right] _{q}!}%
H_{m-n}\left( x|q\right) & if & n\leq m%
\end{array}%
\right. ,
\end{equation*}

iv) $\forall n,m\geq 1:$%
\begin{equation*}
H_{m}\left( x|q\right) B_{n}\left( x|q\right) =(-1)^{n}q^{\binom{n}{2}%
}\sum_{i=0}^{\left\lfloor (n+m)/2\right\rfloor }\QATOPD[ ] {n}{i}_{q}\QATOPD[
] {n+m-i}{i}_{q}\left[ i\right] _{q}!q^{-i(n-i)}H_{n+m-2i}\left( x|q\right) .
\end{equation*}
\end{lemma}

\begin{proof}
i) Follows basically the formula 13.3.6 in \cite{IA} after necessary
re-normalization and re-scaling. iv) Follows i) and (\ref{identity}).
Lengthy, detailed proofs of ii) and iii) are shifted to section \ref{dowody}.
\end{proof}

Since the case $q\allowbreak =\allowbreak 0$ is important to the newly
emerging so called "free probability" (see e.g. nomography \cite{Voi}) let
us see how the considered above sets of polynomials look for $q\allowbreak
=\allowbreak 0$. To do this let us introduce the so called Chebyshev
polynomials of the second kind $U_{n}\left( x\right) $ defined e.g. by the
following three term recurrence :%
\begin{equation}
2xU_{n}\left( x\right) \allowbreak =\allowbreak U_{n+1}\left( x\right)
+U_{n-1}\left( x\right) ,  \label{czeb}
\end{equation}%
for $n\geq 0$ with $U_{-1}\left( x\right) \allowbreak =\allowbreak 0,$ $%
U_{0}\left( x\right) \allowbreak =\allowbreak 1$.

\begin{remark}
\label{q=0}Let us set $q\allowbreak =\allowbreak 0,$ then $S\left( 0\right)
\allowbreak =\allowbreak \lbrack -2,2];\forall n\geq 0,$ we have:

i) $H_{n}\left( x|0\right) \allowbreak =\allowbreak U_{n}\left( x/2\right) ,$

ii) $Q_{n}\left( x|a,b,0\right) \allowbreak =\allowbreak
U_{n}(x)-(a+b)U_{n-1}\left( x\right) \allowbreak +\allowbreak
abU_{n-2}\left( x\right) ,$ $\allowbreak $

iii) $P_{n}\left( x|y,\rho ,0\right) \allowbreak =\allowbreak U_{n}\left(
x/2\right) -\rho yU_{n-1}\left( x/2\right) +\rho ^{2}U_{n-2}\left(
x/2\right) ,$

iv) $B_{-1}\left( y|0\right) \allowbreak =\allowbreak b_{-1}\left(
y|0\right) \allowbreak =\allowbreak 0,$ $B_{0}\left( y|0\right) \allowbreak
=\allowbreak b_{0}\left( y|0\right) \allowbreak =\allowbreak 1,$ $%
B_{n}\left( y|0\right) \allowbreak =\allowbreak \left\{ 
\begin{array}{ccc}
-y & if & n=1 \\ 
1 & if & n=2 \\ 
0 & if & n\geq 3%
\end{array}%
\right. $ and $b_{n}\left( y|0\right) \allowbreak =\allowbreak \left\{ 
\begin{array}{ccc}
-2y & if & n=1 \\ 
1 & if & n=2 \\ 
0 & if & n\geq 3%
\end{array}%
\right. $,

v) $f_{N}\left( x|0\right) \allowbreak =\allowbreak \frac{1}{2\pi }\sqrt{%
4-x^{2}}I_{S\left( 0\right) }$ and 
\begin{equation*}
f_{CN}\left( x|y,\rho ,0\right) \allowbreak =\allowbreak \frac{\left( 1-\rho
^{2}\right) \sqrt{4-x^{2}}}{2\pi w_{0}\left( x,y|\rho ,0\right) }I_{S\left(
0\right) },
\end{equation*}%
for $\left\vert \rho \right\vert <1,$ $y\in S\left( 0\right) $, \newline
vi) 
\begin{equation*}
\phi \left( x|y,\rho _{1},z,\rho _{2},0\right) \allowbreak =\allowbreak 
\frac{\left( 1-\rho _{1}^{2}\right) \left( 1-\rho _{2}^{2}\right)
w_{0}\left( y,z|\rho _{1}\rho _{2},0\right) \sqrt{4-x^{2}}}{\left( 1-\rho
_{1}^{2}\rho _{2}^{2}\right) w_{0}\left( x,y|\rho _{1},0\right) w_{0}\left(
x,z|\rho _{2},0\right) }\frac{1}{2\pi }I_{S\left( 0\right) },
\end{equation*}%
where $w_{0}\left( x,y|\rho _{1},0\right) $ is given by (\ref{w_k}).
\end{remark}

\begin{proof}
To get i) compare (\ref{czeb}) with $x$ replaced by $x/2$ and (\ref{He}) for 
$q=0$. To get ii) again compare (\ref{czeb}) and (\ref{AlSC1}) for $%
q\allowbreak =\allowbreak 0$ and notice that these recursions are the same,
however with different initial values. To get iv) we notice that for $%
q\allowbreak =\allowbreak 0$ and $n\geq 3$ we get $0$. For $n<3$ we get
these values directly from (\ref{_B}). iii) follows iv) and assertion i) of
Proposition \ref{connection} or of course from ii) using (\ref{podstawienie}%
) . To get v) and vi) we insert $q\allowbreak =\allowbreak 0$ in (\ref{qN}),
(\ref{fCN}) and (\ref{_x|yz}).
\end{proof}

\section{Main results\label{glowne}}

We will start this section with the presentation of an alternative form of
the AW polynomials. Let $\left\{ D_{n}\left( x|a,b,c,d,q\right) \right\}
_{n\geq -1}$ be the sequence the AW polynomials such that $D_{n}$ has
coefficient by $x^{n}$ equal to $2^{n}$. Thus the polynomials $\left\{
D_{n}\right\} $ are orthogonal with respect to the density $\psi
(x|a,b,c,d,q)$ mentioned in the Remark \ref{U1}. Let the polynomials $A_{n}$
be defined by the change of variables and parameters by the relationship:%
\begin{equation*}
A_{n}\left( x|y,\rho _{1},z,\rho _{2},q\right) \allowbreak =\allowbreak
D_{n}\left( x\sqrt{1-q}/2|a,b,c,d,q\right) ,
\end{equation*}%
with $a,b,c,d$ related to $y,\rho _{1},z,\rho _{2}$ by (\ref{par1}-\ref{par4}%
). We have:

\begin{theorem}
\label{reprezentacja}i) $\forall n\geq 1:$ 
\begin{equation*}
D_{n}\left( x|a,b,c,d,q\right) \allowbreak =\allowbreak \frac{\left(
ab,cd\right) _{n}}{\left( abcdq^{n-1}\right) _{n}}\sum_{j=0}^{n}\QATOPD[ ] {n%
}{j}_{q}b_{n-j}\left( x|q\right) \sum_{i=0}^{j}\QATOPD[ ] {j}{i}_{q}\frac{%
Q_{i}\left( x|a,b,q\right) Q_{j-i}\left( x|c,d,q\right) }{\left( ab\right)
_{i}\left( cd\right) _{j-i}},
\end{equation*}%
where the polynomials $\left\{ Q_{n}\left( x|a,b,q\right) \right\} $ and $%
\left\{ b_{n}\left( x|q\right) \right\} $ are defined by respectively (\ref%
{AlSC1}) and (\ref{maleB}).

ii) $\forall n\geq 1:$ 
\begin{equation*}
A_{n}\left( x|y,\rho _{1},z,\rho _{2},q\right) =\frac{\left( \rho
_{1}^{2},\rho _{2}^{2}\right) _{n}}{\left( \rho _{1}^{2}\rho
_{2}^{2}q^{n-1}\right) _{n}}\sum_{j=0}^{n}\QATOPD[ ] {n}{j}_{q}B_{n-j}\left(
x|q\right) \sum_{i=0}^{j}\QATOPD[ ] {j}{i}_{q}\frac{P_{i}\left( x|y,\rho
_{1},q\right) P_{j-i}(x|z,\rho _{2},q)}{\left( \rho _{1}^{2}\right)
_{i}\left( \rho _{2}^{2}\right) _{j-i}},
\end{equation*}%
where the polynomials $\left\{ P_{n}(x|y,\rho ,q)\right\} $ and $\left\{
B_{n}\left( x|q\right) \right\} $ are defined by respectively (\ref{AlSC})
and (\ref{_B}).

iii) $\forall n\geq 1:$%
\begin{equation*}
A_{n}\left( x|y,\rho _{1},z,\rho _{2},q\right) \allowbreak \newline
=\allowbreak \frac{\left( \rho _{1}^{2},\rho _{2}^{2}\right) _{n}}{\left(
\rho _{1}^{2}\rho _{2}^{2}q^{n-1}\right) _{n}}\sum_{m=0}^{n}\left( -1\right)
^{m}q^{\binom{m}{2}}\QATOPD[ ] {n}{m}\rho _{1}^{m}\frac{P_{n-m}\left(
x|z,\rho _{2},q\right) P_{m}\left( y|x,\rho _{1},q\right) }{\left( \rho
_{2}^{2}\right) _{n-m}(\rho _{1}^{2})_{m}}.
\end{equation*}
\end{theorem}

\begin{proof}
i) We will use two facts concerning forms of the generating functions of the
polynomials $D_{n}$ and $Q_{n}$. Namely in \cite{IA} (formula 15.2.6) and 
\cite{Koek} (formula 3.1.13) we have the following formula adopted for the
polynomials $D_{n}$ 
\begin{equation*}
\sum_{n\geq 0}\frac{\left( abcdq^{n-1}\right) _{n}D_{n}\left(
x|a,b,c,d,q\right) }{\left( ab,cd,q\right) _{n}}t^{n}=_{~2}\phi _{1}\left(
\QATOPD. \vert {ae^{i\theta },be^{i\theta }}{ab}q,te^{-i\theta }\right)
~_{~2}\phi _{1}\left( \QATOPD. \vert {c^{-i\theta },de^{-i\theta
}}{cd}q,te^{i\theta }\right) ,
\end{equation*}%
where $x\allowbreak =\allowbreak \cos \theta $. On the other hand in \cite%
{Koek} we have the following formula (3.8.14) 
\begin{equation*}
\sum_{n\geq 0}\frac{Q_{n}\left( x|a,b,q\right) }{\left( ab,q\right) _{n}}%
t^{n}\allowbreak =\allowbreak \frac{1}{\left( te^{i\theta }\right) _{\infty }%
}~_{~2}\phi _{1}\left( \QATOPD. \vert {ae^{i\theta },be^{i\theta
}}{ab}q,te^{-i\theta }\right) ,
\end{equation*}%
again with $x\allowbreak =\allowbreak \cos \theta $. Noting that $\cos
(-\theta )\allowbreak =\allowbreak \cos \left( \theta \right) $ we see that%
\begin{eqnarray*}
&&\left( te^{-i\theta },te^{i\theta }\right) _{\infty }\sum_{i\geq 0}\frac{%
Q_{n}\left( x|a,b,q\right) }{\left( ab,q\right) _{n}}t^{n}\sum_{i\geq 0}%
\frac{Q_{n}\left( x|c,d,q\right) }{\left( cd,q\right) _{n}}t^{n}\allowbreak
\\
&=&\allowbreak _{~2}\phi _{1}\left( \QATOPD. \vert {ae^{i\theta
},be^{i\theta }}{ab}q,te^{-i\theta }\right) ~_{~2}\phi _{1}\left( \QATOPD.
\vert {c^{-i\theta },de^{-i\theta }}{cd}q,te^{i\theta }\right) .
\end{eqnarray*}%
Now it remains to notice that $\left( te^{-i\theta },te^{i\theta }\right)
_{\infty }\allowbreak =\allowbreak \prod_{k=0}^{\infty }\left(
1-2xtq^{k}+t^{2}q^{2k}\right) ,$ confront it with the formulae (\ref{_B})
and (\ref{maleB}) and given in \cite{bms} generating function of the
polynomials $B_{n}\left( x|q\right) $ and thus deduce that 
\begin{equation*}
\left( te^{-i\theta },te^{i\theta }\right) _{\infty }\allowbreak
=\allowbreak \sum_{n\geq 0}\frac{b_{n}\left( x|q\right) t^{n}}{\left(
q\right) _{n}}.
\end{equation*}%
Next we apply twice the Cauchy formula for the multiplication of power
series.

ii) Let us change parameters to ones given by (\ref{par1}-\ref{par4}) and
let us also redefine the variable $x$ by introducing instead the variable $%
\xi \allowbreak =\allowbreak 2x/\sqrt{1-q}$ and defining the polynomials $%
A_{n}\left( \xi |y,\rho _{1},z,\rho _{2},q\right) =2^{-n}p_{n}\left(
x,a,b,c,d|q\right) /\left( abcdq^{n-1}\right) _{n}$ where $a,b,c,d$ are
given by (\ref{par1}-\ref{par4}).

iii) The proof of this formula is longer and thus is shifted to section (\ref%
{dowody}).
\end{proof}

As a corollary we get the following property of the ASC polynomials.

\begin{corollary}
\label{symmetry}$\forall n\geq 1:$%
\begin{eqnarray*}
&&\sum_{m=0}^{n}\left( -1\right) ^{m}q^{\binom{m}{2}}\QATOPD[ ] {n}{m}%
_{q}\rho _{1}^{m}\frac{P_{n-m}\left( x|z,\rho _{2},q\right) P_{m}\left(
y|x,\rho _{1},q\right) }{\left( \rho _{2}^{2}\right) _{n-m}(\rho
_{1}^{2})_{m}} \\
&&\sum_{m=0}^{n}\left( -1\right) ^{m}q^{\binom{m}{2}}\QATOPD[ ] {n}{m}%
_{q}\rho _{2}^{m}\frac{P_{n-m}\left( x|y,\rho _{1},q\right) P_{m}\left(
z|x,\rho _{2},q\right) }{\left( \rho _{1}^{2}\right) _{n-m}(\rho
_{2}^{2})_{m}}.
\end{eqnarray*}
\end{corollary}

\begin{proof}
Follows symmetry exposed in assertion ii) of the Theorem.
\end{proof}

\begin{corollary}
\label{free1}For $q\allowbreak =\allowbreak 0$ we get $D_{1}\left(
x|a,b,c,d,0\right) \allowbreak =\allowbreak 2x\allowbreak -\allowbreak \frac{%
a+b+c+d-abc-bcd-acd-abd}{1-abcd},$ $D_{2}\left( x|a,b,c,d,0\right)
\allowbreak =\allowbreak 4x^{2}-2(a+b+c+d)x\allowbreak +\allowbreak
ab\allowbreak +\allowbreak ac+ad\allowbreak +\allowbreak bc\allowbreak
+\allowbreak bd\allowbreak +\allowbreak cd\allowbreak -\allowbreak
1\allowbreak -\allowbreak abcd$ and generally for $n\geq 2$ 
\begin{gather*}
D_{n}\left( x|a,b,c,d,0\right) \allowbreak =\allowbreak \sum_{i=0}^{n}\frac{%
Q_{i}\left( x|a,b,0\right) Q_{n-i}\left( x|c,d,0\right) }{\left( ab;0\right)
_{i}\left( cd;0\right) _{n-i}} \\
-2x\sum_{i=0}^{n-1}\frac{Q_{i}\left( x|a,b,0\right) Q_{n-1-i}\left(
x|c,d,0\right) }{\left( ab;0\right) _{i}\left( cd;0\right) _{n-1-i}} \\
+\sum_{i=0}^{n-2}\frac{Q_{i}\left( x|a,b,0\right) Q_{n-2-i}\left(
x|c,d,0\right) }{\left( ab;0\right) _{i}\left( cd;0\right) _{n-2-i}},
\end{gather*}%
where $Q_{i}\left( x|a,b,0\right) $ and $\left( a;0\right) _{i}$ are defined
by assertion ii) of Remark \ref{q=0} and formulae from the beginning of the
section \ref{pomoc}. Similarly $A_{1}\left( x|y,\rho _{1},z,\rho
_{2},0\right) \allowbreak =\allowbreak x-\frac{y\rho _{1}\left( 1-\rho
_{2}^{2}\right) +z\rho _{2}\left( 1-\rho _{1}^{2}\right) }{1-\rho
_{1}^{2}\rho _{2}^{2}}$ and for $n\geq 2$ 
\begin{gather*}
\frac{A_{n}\left( x|y,\rho _{1},z,\rho _{2},0\right) }{\left( 1-\rho
_{1}^{2}\right) \left( 1-\rho _{2}^{2}\right) }\allowbreak =\allowbreak
\sum_{m=0}^{n}\rho _{1}^{m}\frac{P_{n-m}\left( x|z,\rho _{2},0\right)
P_{m}\left( y|x,\rho _{1},0\right) }{\left( \rho _{2}^{2};0\right)
_{n-m}(\rho _{1}^{2};0)_{m}} \\
-\sum_{m=0}^{n-1}\rho _{1}^{m}\frac{P_{n-1-m}\left( x|z,\rho _{2},0\right)
P_{m}\left( y|x,\rho _{1},0\right) }{\left( \rho _{2}^{2};0\right)
_{n-1-m}(\rho _{1}^{2};0)_{m}},
\end{gather*}%
where $P_{m}\left( x|y,\rho ,0\right) $ are given by assertion iii) of
Remark \ref{q=0}.
\end{corollary}

The main results of the paper concern calculating values of the functions
defined by: 
\begin{equation*}
C_{n}\left( y,z|\rho _{1},\rho _{2},q\right) \allowbreak =\allowbreak
\int_{S\left( q\right) }H_{n}\left( x|q\right) \phi \left( x|y,\rho
_{1},z,\rho _{2},q\right) dx,
\end{equation*}%
$n\geq 1$. These functions have on one hand nice probabilistic
interpretation. Namely assuming that certain $3-$dimensional random vector $%
(Y,X,Z)$ has density equal to $f_{CN}\left( z|x,\rho _{2},q\right)
f_{CN}\left( x|y,\rho _{1},q\right) f_{N}\left( y|q\right) ,$ then 
\begin{equation}
C_{n}\left( y,z|\rho _{1},\rho _{2},q\right) =\mathbb{E}\left( H_{n}\left(
X|q\right) |Y=y,Z=z\right) \allowbreak ,  \label{interpretacja}
\end{equation}%
for almost all (with respect to measure with density $f_{CN}\left( y|z,\rho
_{1}\rho _{2},q\right) f_{N}\left( z|q\right) $) $\left( y,z\right) \in
S\left( q\right) \times S\left( q\right) $. This fact implies in particular
that for almost all $\left( y,z\right) \allowbreak \in \allowbreak S\left(
q\right) \times S\left( q\right) $ we have: $\left\vert C_{n}\left( y,z|\rho
_{1},\rho _{2},q\right) \right\vert \leq \frac{s_{n}\left( q\right) }{\left(
1-q\right) ^{n/2}}.$

\begin{remark}
In \cite{Szab5} it has been shown that functions $C_{n}$ are polynomials in $%
y$ and $z$ of order at most $n$. More precisely it has been shown that 
\begin{equation*}
C_{n}\left( y,z|\rho _{1},\rho _{2},q\right) =\sum_{r=0}^{\left\lfloor
n/2\right\rfloor }\sum_{l=0}^{n-2r}A_{r,-\left\lfloor n/2\right\rfloor
+r+l}^{\left( n\right) }H_{l}\left( y|q\right) H_{n-2r-l}\left( z|q\right) ,
\end{equation*}%
where there are $\left\lfloor \frac{n+2}{2}\right\rfloor \left\lfloor \frac{%
n+3}{2}\right\rfloor $ constants (depending only on $n,$ $q,\rho _{1},\rho
_{2})$ $A_{r,s}^{(n)};$. $r\allowbreak =\allowbreak 0,\ldots ,\left\lfloor
n/2\right\rfloor ,$ $s\allowbreak =\allowbreak -\left\lfloor
n/2\right\rfloor +r,\ldots ,-\left\lfloor n/2\right\rfloor +r\allowbreak
+\allowbreak n-2r$. However the exact general form of these constants was
not found (except for the cases $n=1,2,3,4).$
\end{remark}

As announced in the introduction, in the present paper we will, express the
polynomials $C_{n}$ in terms of the polynomials $H_{n}$ and (or) $P_{n}$.

Namely we will prove the following Theorem:

\begin{theorem}
\label{main}$\forall n\geq 1,\left\vert q\right\vert <1,\left\vert \rho
_{1}\right\vert ,\left\vert \rho _{2}\right\vert <1:$%
\begin{gather}
C_{n}\left( y,z|\rho _{1},\rho _{2},q\right) =\frac{1}{\left( \rho
_{1}^{2}\rho _{2}^{2}\right) _{n}}\sum_{k=0}^{\left\lfloor n/2\right\rfloor
}(-1)^{k}q^{\binom{k}{2}}\QATOPD[ ] {n}{2k}_{q}\QATOPD[ ] {2k}{k}_{q}\left[ k%
\right] _{q}!\times  \label{_C1} \\
\rho _{2}^{2k}\rho _{1}^{2k}\left( \rho _{1}^{2},\rho _{2}^{2}\right)
_{k}\sum_{j=0}^{n-2k}\QATOPD[ ] {n-2k}{j}_{q}\left( \rho
_{1}^{2}q^{k}\right) _{j}\left( \rho _{2}^{2}q^{k}\right) _{n-2k-j}\rho
_{1}^{n-2k-j}\rho _{2}^{j}H_{j}\left( z|q\right) H_{n-2k-j}(y|q).
\label{_C2}
\end{gather}
\end{theorem}

Before presentation of the proof let us make two immediate remarks.

\begin{remark}
Notice that for, say $\rho _{1}\allowbreak =\allowbreak 0$ we get $%
C_{n}\left( y,z|0,\rho _{2},q\right) \allowbreak =\allowbreak \rho
_{2}^{n}H_{n}\left( z|q\right) $ which agrees nicely with the probabilistic
interpretation of the function $C_{n}$ given above. Compare also assertion
ii) of the Proposition \ref{znane}. It is so since $C_{n}\left( y,z|0,\rho
,q\right) \allowbreak =\allowbreak \mathbb{E}\left( H_{n}\left( X|q\right)
|Z=z\right) \allowbreak =\allowbreak \rho ^{n}H_{n}\left( z|q\right) $ a.s.,
($f_{N})$ if $\left( Y,Z\right) \allowbreak \sim \allowbreak f_{CN}\left(
y|z,\rho ,q\right) f_{N}\left( z|q\right) $ as shown in \cite{bryc1}.
\end{remark}

\begin{remark}
Keeping in mind the probabilistic interpretation of the functions $C_{n}$
given in (\ref{interpretacja}), notice that the assertion of Theorem \ref%
{main} enables calculation of all moments of the AW density for complex
parameters. Recently S. Corteel at al. in \cite{Corteel10} announced that
she is going to calculate these moments by some combinatorial methods.
\end{remark}

The proof of this Theorem is based on the following Lemma that in another
form and with the different proof (based heavily on the assertion i) of
Lemma \ref{BH}) was presented in \cite{Szab5}. Notice that assertion i) of
this Lemma is in fact a generalization of an old result of Carlitz \cite%
{Carlitz72} (see also \cite{ALIs88} or partially \cite{IA}, Exercise
12.3(d)). Besides, in this Lemma we present an alternative form of the
function $C_{n}$ this time expressed through polynomials $H_{n}$ and $P_{n}$.

\begin{lemma}
\label{UogCarl}Let us denote $\gamma _{m,k}\left( x,y|\rho ,q\right)
\allowbreak =\allowbreak \sum_{i=0}^{\infty }\frac{\rho ^{i}}{\left[ i\right]
_{q}!}H_{i+m}\left( x|q\right) H_{i+k}\left( y|q\right) \allowbreak $. Then

i) $\gamma _{m,k}\left( x,y|\rho ,q\right) \allowbreak =\allowbreak \gamma
_{0,0}\left( x,y|\rho ,q\right) \sum_{s=0}^{k}(-1)^{s}q^{\binom{s}{2}}\QATOPD%
[ ] {k}{s}_{q}\rho ^{s}H_{k-s}\left( y|q\right) P_{m+s}(x|y,\rho ,q)/(\rho
^{2})_{m+s},$

ii) $C_{n}\left( y,z|\rho _{1},\rho _{2},q\right) \allowbreak =\allowbreak
\sum_{s=0}^{n}\allowbreak \QATOPD[ ] {n}{s}_{q}\rho _{1}^{n-s}\rho
_{2}^{s}\left( \rho _{1}^{2}\right) _{s}H_{n-s}\left( y|q\right) P_{s}\left(
z|y,\rho _{1}\rho _{2},q\right) /(\rho _{1}^{2}\rho _{2}^{2})_{s}$.
\end{lemma}

\begin{proof}
The proof is shifted to section \ref{dowody}.
\end{proof}

As a corollary we get another property of the ASC polynomials:

\begin{corollary}
\label{Al-Salam}$P_{m}\left( y|x,\rho ,q\right) /\left( \rho ^{2}\right)
_{m}\allowbreak =\allowbreak \sum_{s=0}^{m}\left( -1\right) ^{s}\QATOPD[ ] {m%
}{s}_{q}q^{\binom{s}{2}}\rho ^{s}H_{m-s}\left( y|q\right) P_{s}\left(
x|y,\rho ,q\right) /\left( \rho ^{2}\right) _{s}.$
\end{corollary}

\begin{proof}
Note that $\gamma _{m,k}\left( x,y|\rho ,q\right) \allowbreak =\allowbreak
\gamma _{k,m}\left( y,x|\rho ,q\right) $. From assertion ii) of Lemma \ref%
{UogCarl} it follows that on one hand $\gamma _{0,m}\left( x,y|\rho
,q\right) \allowbreak =\allowbreak \gamma _{0,0}\left( x,y|\rho ,q\right)
P_{m}\left( y|x,\rho ,q\right) /\left( \rho ^{2}\right) _{m}$. On the other
hand from assertion i) it follows that \newline
$\gamma _{0,m}\left( x,y|\rho ,q\right) \allowbreak =\allowbreak \gamma
_{0,0}\left( x,y|\rho ,q\right) \sum_{s=0}^{m}(-1)^{s}q^{\binom{s}{2}}\QATOPD%
[ ] {m}{s}_{q}\rho ^{s}H_{m-s}\left( y|q\right) P_{s}(x|y,\rho ,q)/(\rho
^{2})_{s}$.
\end{proof}

As another consequence of Theorem \ref{main} and assertions v) and vii) of
Proposition \ref{znane} we get the following Theorem:

\begin{theorem}
\label{expansion}$\forall -1<q\leq 1,x,y,z\in S\left( q\right) ,|\rho
_{1}|,\left\vert \rho _{2}\right\vert <1,$%
\begin{equation}
\phi \left( x|y,\rho _{1},z,\rho _{2},q\right) =f_{N}\left( x|q\right)
\sum_{i=0}^{\infty }\frac{1}{\left[ i\right] _{q}!}H_{i}\left( x|q\right)
C_{i}\left( y,z|\rho _{1},\rho _{2},q\right) ,  \label{density_exp.}
\end{equation}%
where convergence is absolute and almost uniform on compact sets.
\end{theorem}

\begin{proof}
Is shifted to section \ref{dowody}.
\end{proof}

\section{Open Problems\label{open}}

Notice that $\forall n\geq 1:$ $\int_{S\left( q\right) }(H_{n}\left(
x\right) -C_{n}\left( y,z|\rho _{1},\rho _{2},q\right) )\phi \left( x|y,\rho
_{1},z,\rho _{2},q\right) dx\allowbreak =\allowbreak \int_{S\left( q\right)
}A_{n}\left( x|y,\rho _{1},z,\rho _{2},q\right) \phi \left( x|y,\rho
_{1},z,\rho _{2},q\right) dx\allowbreak =\allowbreak 0$. Hence there must
exist polynomials $F_{n,i}\left( y,z|\rho _{1},\rho _{2},q\right) $ such
that: $\forall n\geq 1:$%
\begin{equation*}
A_{n}\left( x|y,\rho _{1},z,\rho _{2},q\right) \allowbreak =\allowbreak
\sum_{i=1}^{n}F_{n,i}\left( y,z|\rho _{1},\rho _{2},q\right) \left(
H_{i}\left( x\right) -C_{i}\left( y,z|\rho _{1},\rho _{2},q\right) \right) .
\end{equation*}

\begin{enumerate}
\item One would like to find these polynomials.

\item We have $F_{n,n}\left( y,z|\rho _{1},\rho _{2},q\right) \allowbreak
=\allowbreak 1$ (both $\left\{ H_{n}\left( x\right) -C_{n}\left( y,z|\rho
_{1},\rho _{2},q\right) \right\} $ and $\left\{ A_{n}\left( x|y,\rho
_{1},z,\rho _{2},q\right) \right\} $ are monic). When say $\rho
_{2}\allowbreak =\allowbreak 0$ (the ASC case) we have: 
\begin{equation*}
P_{n}\left( x|y,\rho _{1},q\right) \allowbreak =\allowbreak \sum_{i=1}^{n}%
\QATOPD[ ] {n}{i}_{q}\rho _{1}^{n-i}B_{n-i}\left( y|q\right) (H_{i}\left(
x\right) -\rho _{1}^{i}H_{i}\left( y|q\right) ),
\end{equation*}%
which is in fact combination of assertions i) and ii) of Proposition \ref%
{connection}. Thus one would like to ask if the functions $F_{n,i}\left(
y,z|\rho _{1},\rho _{2},q\right) $ also depend on $n-i?$

\item It was shown in \cite{bms} that $\sum_{j=0}^{n}\QATOPD[ ] {n}{j}%
_{q}B_{n-j}\left( y|q\right) H_{j}\left( y|q\right) \allowbreak =\allowbreak
0$ for $y\in S\left( q\right) $ and $n\geq 1.$ Is the same true for the
general case. Namely is it true that: $\forall n\geq 1,$ $y,z\in S\left(
q\right) $ 
\begin{equation*}
\sum_{j=0}^{n}F_{n,j}\left( y,z|\rho _{1},\rho _{2},q\right) C_{i}\left(
y,z|\rho _{1},\rho _{2},q\right) \allowbreak =\allowbreak 0?
\end{equation*}

\item If $q\allowbreak =\allowbreak 1$ we have $\frac{1}{\sqrt{2\pi (1-\rho
^{2})}}\int_{\mathbb{R}}H_{n}\left( x\right) \exp \left( -\frac{(x-\rho m)}{%
2(1-\rho ^{2})}\right) dx\allowbreak =\allowbreak \rho ^{n}H_{n}\left(
m\right) $ hence following observation (\ref{_N}) we deduce that the r\^{o}%
le of the parameter $\rho $ is now played by $\sqrt{\frac{\rho _{1}^{2}+\rho
_{2}^{2}-2\rho _{1}^{2}\rho _{2}^{2}}{1-\rho _{1}^{2}\rho _{2}^{2}}}$ and of 
$m$ by $\frac{y\rho _{1}\left( 1-\rho _{2}^{2}\right) +z\rho _{2}\left(
1-\rho _{1}^{2}\right) }{\sqrt{1-\rho _{1}^{2}\rho _{2}^{2}}\sqrt{\rho
_{1}^{2}+\rho _{2}^{2}-2\rho _{1}^{2}\rho _{2}^{2}}}$. Thus 
\begin{equation}
C_{n}\left( y,z|\rho _{1},\rho _{2},1\right) \allowbreak =\allowbreak \left( 
\sqrt{\frac{\rho _{1}^{2}+\rho _{2}^{2}-2\rho _{1}^{2}\rho _{2}^{2}}{1-\rho
_{1}^{2}\rho _{2}^{2}}}\right) ^{n}H_{n}\left( \frac{y\rho _{1}\left( 1-\rho
_{2}^{2}\right) +z\rho _{2}\left( 1-\rho _{1}^{2}\right) }{\sqrt{1-\rho
_{1}^{2}\rho _{2}^{2}}\sqrt{\rho _{1}^{2}+\rho _{2}^{2}-2\rho _{1}^{2}\rho
_{2}^{2}}}\right) .  \label{c_n_q=1}
\end{equation}%
Is it also true for $\left\vert q\right\vert <1$ with an obvious
modification that $H_{n}\left( x\right) $ is replaced by $H_{n}\left(
x|q\right) $. Most certainly not, but may be $C_{n}\left( y,z|\rho _{1},\rho
_{2},q\right) $ can be presented as a linear combination of expression of
this type, more compact than (\ref{_C1}), (\ref{_C2}). The problem is
connected with the problem of expressing $H_{n}\left( \alpha x+\beta
y|q\right) $ as a linear combination of $H_{i}\left( x|q\right) H_{j}\left(
y|q\right) ,$ $i+j\leq n.$ It has known, nice form for $q\allowbreak
=\allowbreak 1$ and neither nice nor known form for all $n\geq 1$ and other
values of $q.$
\end{enumerate}

\section{Proofs\label{dowody}}

\begin{proof}[Proof of Lemma \protect\ref{nawiasy}]
i) Let us denote $D_{n}\left( a\right) \allowbreak =\allowbreak
\sum_{k=0}^{n}\QATOPD[ ] {n}{k}_{q}\left( a\right) _{n-k}a^{k}\allowbreak $.
Let $\phi \left( t,a\right) \allowbreak =\allowbreak \sum_{n=0}^{\infty }%
\frac{t^{n}}{\left( q\right) _{n}}D_{n}\left( a\right) $ be a characteristic
function of the sequence $\left\{ D_{n}\left( a\right) \right\} $ $%
\allowbreak $. We have $\phi \left( t,a\right) \allowbreak =\allowbreak
\sum_{n=0}^{\infty }\frac{t^{n}}{\left[ n\right] _{q}!}\sum_{i=0}^{n}\QATOPD[
] {n}{i}_{q}\left( a\right) _{i}a_{n-i}\allowbreak =\allowbreak
\sum_{i=0}^{\infty }\frac{t^{i}}{\left( q\right) _{i}}\left( a\right)
_{i}\sum_{n=i}^{\infty }\frac{t^{n-i}}{\left( q\right) _{n-i}}%
a^{n-i}\allowbreak =\allowbreak \frac{1}{\left( at\right) _{\infty }}%
\sum_{i=0}^{\infty }\frac{t^{i}}{\left( q\right) _{i}}\left( a\right)
_{i}\allowbreak =\allowbreak \frac{1}{\left( at\right) _{\infty }}\frac{%
\left( at\right) _{\infty }}{\left( t\right) _{\infty }}\allowbreak
=\allowbreak \frac{1}{\left( t\right) _{\infty }}\allowbreak =\allowbreak
\sum_{n\geq 0}\frac{t^{n}}{\left( q\right) _{n}},$ by $q-$binomial theorem.
So $D_{n}\left( a\right) \allowbreak =\allowbreak 1$. Convergence was for $%
\left\vert q\right\vert ,\left\vert a\right\vert ,\left\vert t\right\vert <1$%
. Thus $D_{n}\left( a\right) $ for $\left\vert a\right\vert <1$ is constant,
but since it is a polynomial we deduce that $D_{n}\left( a\right) $ is
constant for all complex $a.$

ii) Using the expansion formula $\sum_{k=0}^{N}(-1)^{k}\QATOPD[ ] {N}{k}%
_{q}q^{\binom{k}{2}}x^{k}\allowbreak =\allowbreak (x)_{N}$, \newline
$\sum_{i=0}^{n}(-1)^{i}q^{\binom{i}{2}}\QATOPD[ ] {n}{i}_{q}(a)_{i}b^{i}%
\left( abq^{i}\right) _{n-i}\allowbreak =\allowbreak
\sum_{i=0}^{n}(-1)^{i}q^{\binom{i}{2}}\QATOPD[ ] {n}{i}_{q}b^{i}\left(
a\right) _{i}\sum_{k=0}^{n-i}\left( -1\right) ^{k}q^{\binom{k}{2}}\QATOPD[ ]
{n-i}{k}_{q}a^{k}b^{k}q^{ki}\allowbreak =\allowbreak \sum_{s=0}^{n}\left(
-1\right) ^{s}q^{\binom{s}{2}}\QATOPD[ ] {n}{s}_{q}b^{s}\sum_{k=0}^{s}\QATOPD%
[ ] {s}{k}_{q}a^{k}\left( a\right) _{s-k}\allowbreak =$\newline
$\allowbreak \sum_{s=0}^{n}\left( -1\right) ^{s}q^{\binom{s}{2}}\QATOPD[ ] {n%
}{s}_{q}b^{s}\allowbreak =\allowbreak \left( b\right) _{n}$ by i) and the
expansion formula.
\end{proof}

\begin{proof}[Proof of Lemma \protect\ref{BH}.]
ii) First let us recall that by assertion ii) of Proposition \ref{connection}
we have $I_{n,0}(x|q)\allowbreak =\allowbreak 0$ for $n\geq 1$. Next we have 
\begin{equation*}
I_{0,m}(x|q)\allowbreak =\allowbreak H_{m}\left( x|q\right)
,I_{1,m}(x|q)\allowbreak =\allowbreak -xH_{m}(x|q)+H_{m+1}(x|q)\allowbreak
=\allowbreak -[m]_{q}H_{m-1}(x|q).
\end{equation*}%
To prove ii) we apply the formula 
\begin{equation*}
H_{n}\left( x|q\right) H_{m}\left( x|q\right) \allowbreak =\allowbreak
H_{n+m}\left( x|q\right) \allowbreak +\allowbreak \sum_{k=1}^{\min \left(
n,m\right) }\QATOPD[ ] {m}{k}_{q}\QATOPD[ ] {n}{k}_{q}\left[ k\right]
_{q}!H_{n+m-2k}\left( x|q\right)
\end{equation*}%
and get 
\begin{gather*}
I_{n,m}\left( x|q\right) \allowbreak =\sum_{i=0}^{n}\QATOPD[ ] {n}{i}%
_{q}B_{n-i}\left( x|q\right) H_{i+m}\left( x|q\right) \allowbreak \\
=\allowbreak H_{m}\left( x|q\right) I_{n,0}(x|q)\allowbreak -\allowbreak
\sum_{i=1}^{n}\QATOPD[ ] {n}{i}_{q}B_{n-i}\left( x|q\right) \sum_{k=1}^{\min
(i,m)}\QATOPD[ ] {i}{k}_{q}\QATOPD[ ] {m}{k}_{q}\left[ k\right]
_{q}!H_{i+m-2k}\left( x|q\right) \allowbreak \\
=-\allowbreak \sum_{i=1}^{n}\QATOPD[ ] {n}{i}_{q}B_{n-i}\left( x|q\right)
\sum_{k=1}^{\min (i,m)}\QATOPD[ ] {i}{k}_{q}\QATOPD[ ] {m}{k}_{q}\left[ k%
\right] _{q}!H_{i+m-2k}\left( x|q\right) .
\end{gather*}%
After changing the order of summation we get:%
\begin{equation*}
I_{n,m}\left( x|q\right) \allowbreak =-\sum_{k=1}^{n}\QATOPD[ ] {m}{k}_{q}%
\QATOPD[ ] {n}{k}_{q}\left[ k\right] _{q}!\sum_{s=0}^{n-k}\QATOPD[ ] {n-k}{s}%
_{q}B_{n-k-s}\left( x|q\right) H_{s+m-k}\left( x|q\right) .
\end{equation*}%
iii) will be proved by induction with respect to $n$. Let us assume that the
assertion is true for all $n\leq k-1$ . By ii) we have $I_{k,m}\left(
x|q\right) \allowbreak =\allowbreak -\sum_{j=1}^{k}\QATOPD[ ] {m}{j}_{q}%
\QATOPD[ ] {k}{j}_{q}\left[ j\right] _{q}!I_{k-j,m-j}\left( x|q\right) $.
Now if $m<k$ we see that then $k-j<m-j$ for all $j=1,\ldots ,k$ and thus by
induction $I_{k-j,m-j}\left( x|q\right) \allowbreak =\allowbreak 0$. If $%
k\geq m$ then by the induction assumption we have $I_{k-j,m-j}\left(
x|q\right) \allowbreak =\allowbreak (-1)^{k-j}q^{\binom{k-j}{2}}\frac{\left[
m-j\right] _{q}!}{\left[ m-k\right] _{q}!}H_{m-k}\left( x|q\right) $. Hence 
\begin{gather*}
I_{k,m}\left( x|q\right) \allowbreak =\allowbreak -\sum_{j=1}^{k}\QATOPD[ ] {%
m}{j}_{q}\QATOPD[ ] {k}{j}_{q}\left[ j\right] _{q}(-1)^{k-j}q^{\binom{k-j}{2}%
}\frac{\left[ m-j\right] _{q}!}{\left[ m-k\right] _{q}!}H_{m-k}\left(
x|q\right) \allowbreak \\
=\allowbreak -\frac{\left[ m\right] _{q}!}{\left[ m-k\right] _{q}!}%
H_{m-k}\left( x|q\right) \sum_{j=1}^{k}\QATOPD[ ] {k}{j}_{q}(-1)^{k-j}q^{%
\binom{k-j}{2}}\allowbreak = \\
\allowbreak -\frac{\left[ m\right] _{q}!}{\left[ m-k\right] _{q}!}%
H_{m-k}\left( x|q\right) \sum_{s=0}^{k-1}\QATOPD[ ] {k}{s}_{q}(-1)^{s}q^{%
\binom{s}{2}}\allowbreak \\
=\allowbreak -\frac{\left[ m\right] _{q}!}{\left[ m-k\right] _{q}!}%
H_{m-k}\left( x|q\right) \allowbreak (\sum_{s=0}^{k-1}\QATOPD[ ] {k}{s}%
_{q}(-1)^{s}q^{\binom{s}{2}}\allowbreak +\allowbreak (-1)^{k}q^{\binom{k}{2}%
}\allowbreak -\allowbreak (-1)^{k}q^{\binom{k}{2}})\allowbreak \\
=\allowbreak (-1)^{k}q^{\binom{k}{2}}\frac{\left[ m\right] _{q}!}{\left[ m-k%
\right] _{q}!}H_{m-k}\left( x|q\right) ,
\end{gather*}%
since $\sum_{s=0}^{k-1}\QATOPD[ ] {k}{s}_{q}(-1)^{s}q^{\binom{s}{2}%
}\allowbreak +\allowbreak (-1)^{k}q^{\binom{k}{2}}\allowbreak =\allowbreak
(1)_{k}\allowbreak =\allowbreak 0$.
\end{proof}

\begin{proof}[Proof of assertion iii) of Theorem \protect\ref{reprezentacja}]

We start with the assertion of Corollary \ref{Al-Salam} and the assertion
iii) of Proposition \ref{znane}. Using them we get:%
\begin{eqnarray*}
&&\int_{S\left( q\right) }P_{m}\left( z|y,t,q\right) P_{k}\left(
y|z,t,q\right) f_{CN}\left( z|y,t,q\right) dz\allowbreak \\
&=&\allowbreak \left\{ 
\begin{array}{ccc}
0 & if & m>k \\ 
(-1)^{m}q^{\binom{m}{2}}\frac{\left[ k\right] _{q}!}{\left[ k-m\right] _{q}!}%
t^{m}H_{k-m}\left( y|q\right) \left( t^{2}\right) _{k} & if & m\leq k%
\end{array}%
\right. .
\end{eqnarray*}%
Using the assertion ii) of Theorem \ref{reprezentacja} let us calculate 
\newline
$V_{n,m}\left( x,z,\rho _{1},\rho _{2}|q\right) \allowbreak =\allowbreak
\int_{S\left( q\right) }A_{n}\left( x|y,\rho _{1},z,\rho _{2},q\right)
P_{m}\left( y|x,\rho _{1},q\right) f_{CN}\left( y|x,\rho _{1},q\right) dy$. $%
\allowbreak $We have%
\begin{gather*}
V_{n,m}\left( x,z,\rho _{1},\rho _{2}|q\right) =\frac{\left( \rho
_{1}^{2},\rho _{2}^{2}\right) _{n}}{\left( \rho _{1}^{2}\rho
_{2}^{2}q^{n-1}\right) _{n}}(-1)^{m}q^{\binom{m}{2}}\rho
_{1}^{m}\sum_{j=0}^{n}\QATOPD[ ] {n}{j}_{q}B_{n-j}\left( x|q\right) \\
\times \sum_{i=m}^{j}\QATOPD[ ] {j}{i}_{q}\frac{P_{j-i}\left( x|z,\rho
_{2},q\right) }{\left( \rho _{2}^{2}\right) _{j-i}}\frac{\left[ i\right]
_{q}!}{\left[ i-m\right] _{q}!}H_{i-m}\left( x|q\right) \\
=\frac{\left( \rho _{1}^{2},\rho _{2}^{2}\right) _{n}}{\left( \rho
_{1}^{2}\rho _{2}^{2}q^{n-1}\right) _{n}}(-1)^{m}q^{\binom{m}{2}}\rho
_{1}^{m}\sum_{j=m}^{n}\QATOPD[ ] {n}{j}_{q}B_{n-j}\left( x|q\right) \\
\times \sum_{k=0}^{j-m}\frac{\left[ j\right] _{q}!}{\left[ j-m-k\right] _{q}!%
\left[ k\right] _{q}!}\frac{P_{j-m-k}\left( x|z,\rho _{2},q\right) }{\left(
\rho _{2}^{2}\right) _{j-m-k}}H_{k}\left( x|q\right) \\
=\allowbreak \frac{\left( \rho _{1}^{2},\rho _{2}^{2}\right) _{n}}{\left(
\rho _{1}^{2}\rho _{2}^{2}q^{n-1}\right) _{n}}(-1)^{m}q^{\binom{m}{2}}\rho
_{1}^{m}\frac{\left[ n\right] _{q}!}{\left[ n-m\right] _{q}!}%
\sum_{s=0}^{n-m}B_{n-m-s}\left( x|q\right) \frac{\left[ n-m\right] _{q}!}{%
\left[ n-m-s\right] _{q}!\left[ s\right] _{q}!}\times \\
\allowbreak \sum_{k=0}^{s}\frac{\left[ s\right] _{q}!}{\left[ s-k\right]
_{q}!\left[ k\right] _{q}!}\frac{P_{s-k}\left( x|z,\rho _{2},q\right) }{%
\left( \rho _{2}^{2}\right) _{s-k}}H_{k}\left( x|q\right) .
\end{gather*}%
We change the order of summation and get%
\begin{gather*}
V_{n,m}\left( x,z,\rho _{1},\rho _{2}|q\right) =\frac{\left( \rho
_{1}^{2},\rho _{2}^{2}\right) _{n}}{\left( \rho _{1}^{2}\rho
_{2}^{2}q^{n-1}\right) _{n}}(-1)^{m}q^{\binom{m}{2}}\rho _{1}^{m}\frac{\left[
n\right] _{q}!}{\left[ n-m\right] _{q}!}\sum_{k=0}^{n-m}\QATOPD[ ] {n-m}{k}%
_{q}H_{k}(x|q)\times \\
\sum_{s=k}^{n-m}\QATOPD[ ] {n-m-k}{s-k}_{q}\frac{P_{s-k}\left( x|z,\rho
_{2},q\right) }{\left( \rho _{2}^{2}\right) _{s-k}}B_{n-m-s}\left( x|q\right)
\\
=\frac{\left( \rho _{1}^{2},\rho _{2}^{2}\right) _{n}}{\left( \rho
_{1}^{2}\rho _{2}^{2}q^{n-1}\right) _{n}}(-1)^{m}q^{\binom{m}{2}}\rho
_{1}^{m}\frac{\left[ n\right] _{q}!}{\left[ n-m\right] _{q}!}\sum_{k=0}^{n-m}%
\QATOPD[ ] {n-m}{k}_{q}H_{k}(x|q)\times \\
\sum_{j=0}^{n-m-k}\QATOPD[ ] {n-m-k}{j}_{q}\frac{P_{j}\left( x|z,\rho
_{2},q\right) }{\left( \rho _{2}^{2}\right) _{j}}B_{n-m-k-j}\left( x|q\right)
\\
=\frac{\left( \rho _{1}^{2},\rho _{2}^{2}\right) _{n}}{\left( \rho
_{1}^{2}\rho _{2}^{2}q^{n-1}\right) _{n}}(-1)^{m}q^{\binom{m}{2}}\rho
_{1}^{m}\frac{\left[ n\right] _{q}!}{\left[ n-m\right] _{q}!}\sum_{j=0}^{n-m}%
\QATOPD[ ] {n-m}{j}_{q}\frac{P_{j}\left( x|z,\rho _{2},q\right) }{\left(
\rho _{2}^{2}\right) _{j}}\times \\
\sum_{k=0}^{n-m-j}\QATOPD[ ] {n-m-j}{k}_{q}H_{k}(x|q)B_{n-m-k-j}\left(
x|q\right) .
\end{gather*}%
Now we use the assertion iii) of Lemma \ref{BH} and deduce that $\allowbreak
\sum_{k=0}^{n-m-j}\QATOPD[ ] {n-m-j}{k}_{q}H_{k}(x|q)B_{n-m-k-j}\left(
x|q\right) =\allowbreak 0$ if only $n-m-j>0$. and $1$ if $j\allowbreak
=\allowbreak n-m$. Hence $V_{n,m}\left( x,z,\rho _{1},\rho _{2}|q\right)
\allowbreak =\allowbreak \frac{\left( \rho _{1}^{2},\rho _{2}^{2}\right) _{n}%
}{\left( \rho _{1}^{2}\rho _{2}^{2}q^{n-1}\right) _{n}}(-1)^{m}q^{\binom{m}{2%
}}\rho _{1}^{m}\frac{\left[ n\right] _{q}!}{\left[ n-m\right] _{q}!}\frac{%
P_{n-m}\left( x|z,\rho _{2},q\right) }{\left( \rho _{2}^{2}\right) _{n-m}}$.
Keeping in mind the assertion iii of Proposition \ref{znane} and the
interpretation of $V_{n,m}$ we get \newline
$A_{n}\left( x|y,\rho _{1},z,\rho _{2},q\right) \allowbreak =\frac{\left(
\rho _{1}^{2},\rho _{2}^{2}\right) _{n}}{\left( \rho _{1}^{2}\rho
_{2}^{2}q^{n-1}\right) _{n}}\sum_{m=0}^{n}\QATOPD[ ] {n}{m}_{q}\left(
-1\right) ^{m}q^{\binom{m}{2}}\rho _{1}^{m}\frac{P_{n-m}\left( x|z,\rho
_{2},q\right) P_{m}\left( y|x,\rho _{1},q\right) }{\left( \rho
_{2}^{2}\right) _{n-m}(\rho _{1}^{2})_{m}}$.
\end{proof}

\begin{proof}[Proof of Lemma \protect\ref{UogCarl}]
i) First notice that $\gamma _{0,0}\left( x,y|\rho ,q\right) f_{N}\left(
x|q\right) \allowbreak =\allowbreak f_{CN}\left( x|y,\rho ,q\right) $
(compare \ref{P-M}). Besides we will use assertions i) and ii) of
Proposition \ref{znane}. Since for $\forall x,y\in S\left( q\right) ,$ $%
\gamma _{0,0}\left( x,y|\rho ,q\right) >0$ we can write $\int_{S\left(
q\right) }P_{n}\left( x|y,\rho ,q\right) \gamma _{m,k}\left( x,y|\rho
,q\right) f_{N}\left( x|q\right) dx=\allowbreak $

$=\allowbreak \int_{S\left( q\right) }P_{n}\left( x|y,\rho ,q\right) \frac{%
\gamma _{m,k}\left( x,y|\rho ,q\right) }{\gamma _{0,0}(x,y|\rho ,q)}%
f_{CN}\left( x|y,\rho ,q\right) dx$. \newline
Now 
\begin{eqnarray*}
&&\int_{S\left( q\right) }P_{n}\left( x|y,\rho ,q\right) \gamma _{m,k}\left(
x,y|\rho ,q\right) f_{N}\left( x|q\right) dx\allowbreak \\
&=&\allowbreak \sum_{i\geq 0}\frac{\rho ^{i}}{\left[ i\right] _{q}!}%
H_{i+k}\left( y|q\right) \int_{S\left( q\right) }P_{n}\left( x|y,\rho
,q\right) H_{i+m}\left( x|q\right) f_{N}\left( x|q\right) dx.
\end{eqnarray*}%
Let us recall the assertion i) of Proposition \ref{connection}. Hence we
have 
\begin{eqnarray*}
&&\int_{S\left( q\right) }P_{n}\left( x|y,\rho ,q\right) \gamma _{m,k}\left(
x,y|\rho ,q\right) f_{N}\left( x|q\right) dx\allowbreak \\
&=&\allowbreak \sum_{i\geq 0}\frac{\rho ^{i}}{\left[ i\right] _{q}!}%
H_{i+k}\left( y|q\right) \sum_{j=0}^{n}\QATOPD[ ] {n}{j}_{q}\rho
^{n-j}B_{n-j}\left( y|q\right) \int_{S\left( q\right) }H_{j}\left(
x|q\right) H_{i+m}\left( x|q\right) f_{N}\left( x|q\right) dx\allowbreak .
\end{eqnarray*}%
Obviously if $m>n\allowbreak $ we get $0$. Otherwise when $n\geq m$ we
obtain: 
\begin{eqnarray*}
&&\int_{S\left( q\right) }P_{n}\left( x|y,\rho ,q\right) \gamma _{m,k}\left(
x,y|\rho ,q\right) f_{N}\left( x|q\right) dx\allowbreak \\
&=&\allowbreak \frac{\left[ n\right] _{q}!\rho ^{n-m}}{\left[ n-m\right]
_{q}!}\sum_{i=0}^{n-m}\frac{\left[ n-m\right] _{q}!}{\left[ i\right]
_{q}![n-i-m]}H_{i+k}\left( y|q\right) B_{n-i-m}(y|q)\allowbreak \\
&=&\allowbreak \frac{\left[ n\right] _{q}!\rho ^{n-m}}{\left[ n-m\right]
_{q}!}I_{n-m,k}\left( y|q\right) \allowbreak =\allowbreak (-1)^{n-m}q^{%
\binom{n-m}{2}}\frac{\left[ n\right] _{q}!\rho ^{n-m}\left[ k\right] _{q}!}{%
\left[ n-m\right] _{q}!\left[ k+m-n\right] _{q}!}H_{k+m-n}(y|q).
\end{eqnarray*}%
Hence $\frac{\gamma _{m,k}\left( x,y|\rho ,q\right) }{\gamma
_{0,0}(x,y,|\rho ,q)}\allowbreak =\allowbreak \sum_{n=m}^{m+k}(-1)^{n-m}q^{%
\binom{n-m}{2}}\rho ^{n-m}\QATOPD[ ] {k}{n-m}_{q}H_{k-(n-m)}\left(
y|q\right) P_{n}\left( x|y,\rho ,q\right) /\left( \rho ^{2}\right) _{n}$ or
equivalently $\frac{\gamma _{m,k}\left( x,y|\rho ,q\right) }{\gamma
_{0,0}(x,y|\rho ,q)}\allowbreak =\allowbreak \sum_{s=0}^{k}(-1)^{s}q^{\binom{%
s}{2}}\QATOPD[ ] {k}{s}_{q}\rho ^{s}H_{k-s}\left( y|q\right)
P_{m+s}(x|y,\rho ,q)/(\rho ^{2})_{m+s}$.

ii) We have 
\begin{gather*}
C_{n}\left( x,y|\rho _{1},\rho _{2},q\right) \allowbreak =\allowbreak 
\newline
\frac{1}{\gamma _{0,0}\left( x,y,\rho _{1}\rho _{2}|q\right) }\sum_{i=0}^{n}%
\QATOPD[ ] {n}{i}_{q}\rho _{1}^{n-i}\rho _{2}^{i}\gamma _{i,n-i}\left(
x,y,\rho _{1}\rho _{2}|q\right) \allowbreak \\
=\allowbreak \frac{1}{\gamma _{0,0}\left( x,y,\rho _{1}\rho _{2}|q\right) }%
\sum_{i=0}^{n}\QATOPD[ ] {n}{i}_{q}\rho _{1}^{n-i}\rho
_{2}^{i}\sum_{j=0}^{n-i}(-1)^{j}\QATOPD[ ] {n-i}{j}_{q}q^{\binom{j}{2}}\rho
_{1}^{j}\rho _{2}^{j}H_{n-i-j}\left( y|q\right) P_{i+j}\left( x|y,\rho
_{1}\rho _{2},q\right) /(\rho _{1}^{2}\rho _{2}^{2})_{i+j}\allowbreak \\
=\allowbreak \sum_{s=0}^{n}\QATOPD[ ] {n}{s}_{q}H_{n-s}\left( y|q\right)
P_{s}\left( x|y,\rho _{1}\rho _{2},q\right) /(\rho _{1}^{2}\rho
_{2}^{2})_{s}\sum_{j=0}^{s}(-1)^{j}\QATOPD[ ] {s}{j}_{q}q^{\binom{j}{2}}\rho
_{1}^{j}\rho _{2}^{j}\rho _{1}^{n-s+j}\rho _{2}^{s-j}.
\end{gather*}%
$\allowbreak $ Now using formula (12.2.27) of \cite{IA}, that is $\left(
a\right) _{n}\allowbreak =\allowbreak \sum_{k=0}^{n}\left( -1\right) ^{k}%
\QATOPD[ ] {n}{k}_{q}q^{\binom{k}{2}}a^{k}$ we get $C_{n}\left( x,y|\rho
_{1},\rho _{2},q\right) \allowbreak =\allowbreak \sum_{s=0}^{n}\allowbreak 
\QATOPD[ ] {n}{s}_{q}\rho _{1}^{n-s}\rho _{2}^{s}\left( \rho _{1}^{2}\right)
_{s}H_{n-s}\left( y|q\right) P_{s}\left( x|y,\rho _{1}\rho _{2},q\right)
/(\rho _{1}^{2}\rho _{2}^{2})_{s}$
\end{proof}

\begin{proof}[Proof of Theorem \protect\ref{main}]
\begin{gather*}
C_{n}\left( y,z|\rho _{1},\rho _{2},q\right) \allowbreak =\allowbreak
\sum_{s=0}^{n}\allowbreak \QATOPD[ ] {n}{s}_{q}\rho _{1}^{n-s}\rho
_{2}^{s}\left( \rho _{1}^{2}\right) _{s}H_{n-s}\left( y|q\right) P_{s}\left(
z|y,\rho _{1}\rho _{2},q\right) /(\rho _{1}^{2}\rho _{2}^{2})_{s}\allowbreak
\\
=\sum_{s=0}^{n}\allowbreak \QATOPD[ ] {n}{s}_{q}\rho _{1}^{n-s}\rho
_{2}^{s}\left( \rho _{1}^{2}\right) _{s}H_{n-s}\left( y|q\right)
\sum_{j=0}^{s}\QATOPD[ ] {s}{j}_{q}\rho _{1}^{s-j}\rho
_{2}^{s-j}B_{s-j}\left( y|q\right) H_{j}\left( z|q\right) /\left( \rho
_{1}^{2}\rho _{2}^{2}\right) _{s} \\
=\frac{1}{\left( \rho _{1}^{2}\rho _{2}^{2}\right) _{n}}\sum_{j=0}^{n}\QATOPD%
[ ] {n}{j}_{q}\rho _{1}^{n-j}H_{j}\left( z|q\right) \sum_{s=j}^{n}\QATOPD[ ]
{n-j}{s-j}_{q}\left( \rho _{1}^{2}\right) _{s}\rho _{2}^{2s-j}\left( \rho
_{1}^{2}\rho _{2}^{2}q^{s}\right) _{n-s}B_{s-j}\left( y|q\right)
H_{n-s}\left( y|q\right) \\
\frac{1}{\left( \rho _{1}^{2}\rho _{2}^{2}\right) _{n}}\sum_{j=0}^{n}\QATOPD[
] {n}{j}_{q}\rho _{1}^{n-j}\rho _{2}^{j}H_{j}\left( z|q\right)
\sum_{m=0}^{n-j}\QATOPD[ ] {n-j}{m}_{q}\left( \rho _{1}^{2}\right)
_{m+j}\rho _{2}^{2m}\left( \rho _{1}^{2}\rho _{2}^{2}q^{m+j}\right)
_{n-j-m}B_{m}\left( y|q\right) H_{n-j-m}\left( y|q\right) .
\end{gather*}%
Now we apply the formula given in the assertion iv of Lemma \ref{BH} getting 
\begin{gather*}
C_{n}\left( y,z|\rho _{1},\rho _{2},q\right) \allowbreak =\frac{1}{\left(
\rho _{1}^{2}\rho _{2}^{2}\right) _{n}}\QATOPD[ ] {n}{j}_{q}\rho
_{1}^{n-j}\rho _{2}^{j}\left( \rho _{1}^{2}\right) _{j}H_{j}\left( z|q\right)
\\
\times \sum_{m=0}^{n-j}\QATOPD[ ] {n-j}{m}_{q}\left( \rho
_{1}^{2}q^{j}\right) _{m}\rho _{2}^{2m}\left( \rho _{1}^{2}\rho
_{2}^{2}q^{m+j}\right) _{n-j-m} \\
\times \left( -1\right) ^{m}q^{\binom{m}{2}}\sum_{k=0}^{\left\lfloor
(n-j)/2\right\rfloor }\QATOPD[ ] {m}{k}_{q}\QATOPD[ ] {n-j-k}{k}_{q}\left[ k%
\right] _{q}!q^{-k(m-k)}H_{n-j-2k}(y|q)\allowbreak .
\end{gather*}%
Now we notice that $\QATOPD[ ] {m}{k}_{q}\allowbreak =\allowbreak 0$ if $k>m$%
. So we split the range of $m$ into two subranges $0,\ldots ,\left\lfloor
(n-j)/2\right\rfloor $ and $\left\lfloor (n-j)/2\right\rfloor +1,\ldots ,n-j$%
. Thus the second sum can be transformed in the following way:%
\begin{gather*}
\sum_{m=0}^{\left\lfloor (n-j)/2\right\rfloor }\QATOPD[ ] {n-j}{m}_{q}\left(
\rho _{1}^{2}q^{j}\right) _{m}\rho _{2}^{2m}\left( \rho _{1}^{2}\rho
_{2}^{2}q^{m+j}\right) _{n-j-m}\left( -1\right) ^{m}q^{\binom{m}{2}}\times \\
\sum_{k=0}^{m}\QATOPD[ ] {m}{k}_{q}\QATOPD[ ] {n-j-k}{k}_{q}\left[ k\right]
_{q}!q^{-k(m-k)}H_{n-j-2k}(y|q)+ \\
\sum_{m=\left\lfloor (n-j)/2\right\rfloor +1}^{n-j}\QATOPD[ ] {n-j}{m}%
_{q}\left( \rho _{1}^{2}q^{j}\right) _{m}\rho _{2}^{2m}\left( \rho
_{1}^{2}\rho _{2}^{2}q^{m+j}\right) _{n-j-m}\times \\
\left( -1\right) ^{m}q^{\binom{m}{2}}\sum_{k=0}^{\left\lfloor
(n-j)/2\right\rfloor }\QATOPD[ ] {m}{k}_{q}\QATOPD[ ] {n-j-k}{k}_{q}\left[ k%
\right] _{q}!q^{-k(m-k)}H_{n-j-2k}(y|q).
\end{gather*}%
Now after changing the order of summation we obtain:%
\begin{gather*}
\sum_{k=0}^{\left\lfloor (n-j)/2\right\rfloor }\QATOPD[ ] {n-j-k}{k}_{q}%
\left[ k\right] _{q}!H_{n-j-2k}(y|q)\times \\
\sum_{m=k}^{\left\lfloor (n-j)/2\right\rfloor }\allowbreak \left( -1\right)
^{m}q^{\binom{m}{2}}q^{-k(m-k)}\QATOPD[ ] {n-j}{m}_{q}\QATOPD[ ] {m}{k}%
_{q}\left( \rho _{1}^{2}q^{j}\right) _{m}\rho _{2}^{2m}\left( \rho
_{1}^{2}\rho _{2}^{2}q^{m+j}\right) _{n-j-m} \\
+\sum_{k=0}^{\left\lfloor (n-j)/2\right\rfloor }\QATOPD[ ] {n-j-k}{k}_{q}%
\left[ k\right] _{q}!H_{n-j-2k}(y|q)\times \\
\sum_{m=\left\lfloor (n-j)/2\right\rfloor +1}^{n-j}\left( -1\right) ^{m}q^{%
\binom{m}{2}}q^{-k(m-k)}\QATOPD[ ] {n-j}{m}_{q}\QATOPD[ ] {m}{k}_{q}\left(
\rho _{1}^{2}q^{j}\right) _{m}\rho _{2}^{2m}\left( \rho _{1}^{2}\rho
_{2}^{2}q^{m+j}\right) _{n-j-m}) \\
=\sum_{k=0}^{\left\lfloor (n-j)/2\right\rfloor }\QATOPD[ ] {n-j-k}{k}_{q}%
\left[ k\right] _{q}!H_{n-j-2k}(y|q)\times \\
\sum_{m=k}^{n-j}\allowbreak \left( -1\right) ^{m}q^{\binom{m}{2}}q^{-k(m-k)}%
\QATOPD[ ] {n-j}{m}_{q}\QATOPD[ ] {m}{k}_{q}\left( \rho _{1}^{2}q^{j}\right)
_{m}\rho _{2}^{2m}\left( \rho _{1}^{2}\rho _{2}^{2}q^{m+j}\right) _{n-j-m}.
\end{gather*}%
After changing in the last sum the variable $m$ ranging from $k,\ldots ,m-j$
to $s$ ranging from $0$ to $n-j-k$ and applying firstly formula $\binom{s+k}{%
2}-sk\allowbreak =\allowbreak \binom{s}{2}\allowbreak +\allowbreak \binom{k}{%
2},$ then formula $(a)_{n+m}\allowbreak =\allowbreak \left( a\right)
_{n}\left( aq^{n}\right) _{m}$ and finally assertion ii) of Lemma \ref%
{nawiasy} we get

\begin{eqnarray*}
C_{n}\left( y,z|\rho _{1},\rho _{2},q\right) &=&\frac{1}{\left( \rho
_{1}^{2}\rho _{2}^{2}\right) _{n}}\sum_{j=0}^{n}\QATOPD[ ] {n}{j}_{q}\rho
_{1}^{n-j}\rho _{2}^{j}H_{j}\left( z|q\right) \allowbreak \times \\
&&\sum_{k=0}^{\left\lfloor (n-j)/2\right\rfloor }(-1)^{k}q^{\binom{k}{2}%
}\rho _{2}^{2k}\left( \rho _{1}^{2}\right) _{k+j}\left( \rho _{2}^{2}\right)
_{n-j-k}\frac{\left[ n-j\right] _{q}!}{\left[ n-j-2k\right] _{q}!}%
H_{n-j-2k}(y|q).
\end{eqnarray*}%
Now we change again the order of summing, applying formulae $\left( a\right)
_{n+m}\allowbreak =\allowbreak \left( a\right) _{n}\left( aq^{n}\right) _{m}$
applied to $\left( \rho _{1}^{2}\right) _{k+j}$ and $\left( \rho
_{2}^{2}\right) _{n-j-k}$ we get%
\begin{eqnarray*}
C_{n}(y,z|\rho _{1},\rho _{2},q) &=&\frac{1}{\left( \rho _{1}^{2}\rho
_{2}^{2}\right) _{n}}\sum_{k=0}^{\left\lfloor n/2\right\rfloor }(-1)^{k}q^{%
\binom{k}{2}}\QATOPD[ ] {n}{2k}_{q}\QATOPD[ ] {2k}{k}_{q}\left[ k\right]
_{q}!\rho _{2}^{2k}\rho _{1}^{2k}\left( \rho _{1}^{2},\rho _{2}^{2}\right)
_{k}\times \\
&&\sum_{j=0}^{n-2k}\QATOPD[ ] {n-2k}{j}_{q}\left( \rho _{1}^{2}q^{k}\right)
_{j}\left( \rho _{2}^{2}q^{k}\right) _{n-j-2k}\rho _{1}^{n-2k-j}\rho
_{2}^{j}H_{j}\left( z|q\right) H_{n-j-2k}(y|q).
\end{eqnarray*}
\end{proof}

\begin{proof}[Proof of Theorem \protect\ref{expansion}]
For $\left\vert q\right\vert <1$ we use the assertion vii) of Proposition %
\ref{znane} and Remark \ref{U1} and deduce that $\phi \left( x|y,\rho
_{1},z,\rho _{2},q\right) /f_{N}\left( x|q\right) $ is bounded on $S\left(
q\right) $ hence square integrable with respect to the measure with density $%
f_{N}\left( x|q\right) ,$ thus immediately we get $L_{2}$ convergence in (%
\ref{density_exp.}). To get almost sure convergence let us notice that $\phi
\left( x|y,\rho _{1},z,\rho _{2},q\right) /f_{N}\left( x|q\right) $ is also
square integrable with respect to the measure that has density equal to $%
f_{N}\left( x|q\right) f_{N}\left( y|q\right) f_{N}\left( z|q\right) .$ Next
we notice that polynomials $\left\{ H_{i}\left( x|q\right) H_{j}\left(
y|q\right) H_{k}\left( z|q\right) \right\} _{i,j,k\geq 0}$ constitute an
orthogonal basis of the space $(S\left( q\right) \allowbreak \times
\allowbreak S\left( q\right) \allowbreak \times \allowbreak S\left( q\right)
,\allowbreak \mathcal{B},\allowbreak f_{N}\left( x|q\right) f_{N}\left(
y|q\right) f_{N}\left( z|q\right) ),$ where $\mathcal{B}$ denotes $\sigma -$%
field of Borel subsets of $S\left( q\right) \allowbreak \times \allowbreak
S\left( q\right) \allowbreak \times \allowbreak S\left( q\right) .$ Moreover
we know Fourier coefficients of expansion of $\phi \left( x|y,\rho
_{1},z,\rho _{2},q\right) /f_{N}\left( x|q\right) $ in this basis. Namely we
can read them from expansion (\ref{_C1},\ref{_C2}). They are equal to 
\begin{gather*}
\alpha _{n,j,m}=\int_{S^{3}\left( q\right) }H_{n}\left( x|q\right)
H_{j}\left( y|q\right) H_{m}\left( z|q\right) \phi \left( x|y,\rho
_{1},z,\rho _{2},q\right) f_{N}\left( y|q\right) f_{N}\left( z|q\right)
dxdydz= \\
=\left\{ 
\begin{array}{ccc}
0 & if & j+m\geq n\vee n-j-m\text{ is odd} \\ 
\left( -1\right) ^{k}\frac{q^{\binom{k}{2}}\rho _{1}^{n-j}\left( \rho
^{2}\right) _{j+k}\rho _{2}^{n-m}\left( \rho _{2}^{2}\right) _{n-j-k}}{\left[
k\right] _{q}!\left( \rho _{1}^{2}\rho _{2}^{2}\right) _{n}} & if & n-j-m=2k%
\end{array}%
\right. .
\end{gather*}%
From the theory of the orthogonal series expansions it follows that $%
\sum_{n,j,m}\alpha _{n,j,m}^{2}<\infty $, moreover one can see these
coefficients decrease geometrically. \newline
Hence $\sum_{n,j,m}\alpha _{n,j,m}^{2}\left( \log n\log j\log m\right)
^{2}<\infty $ and thus form the Rademacher-Menshov theorem we get almost
everywhere convergence of the series:%
\begin{equation*}
\sum_{n,j,m\geq 0}\frac{\alpha _{n,j,m}}{\left[ n\right] _{q}!\left[ j\right]
_{q}!\left[ m\right] _{q}!}H_{n}\left( x|q\right) H_{j}\left( y|q\right)
H_{m}\left( z|q\right) .
\end{equation*}%
On the other hand after regrouping nonzero summands of this series we get (%
\ref{density_exp.}).

For $q\allowbreak =\allowbreak 1$ we deal with the normal case. In this case
the functions $C_{n}$ have special form given by (\ref{c_n_q=1}). Thus we
deal with summing of special form of a classical Poisson--Mehler kernel%
\begin{equation*}
\sum_{n\geq 0}\frac{t^{n}}{n!}H_{n}\left( x\right) H_{n}\left( u\right) ,
\end{equation*}%
where $t\allowbreak =\allowbreak \sqrt{\frac{\rho _{1}^{2}+\rho
_{2}^{2}-2\rho _{1}^{2}\rho _{2}^{2}}{1-\rho _{1}^{2}\rho _{2}^{2}}}$ and $%
u\allowbreak =\allowbreak \frac{y\rho _{1}\left( 1-\rho _{2}^{2}\right)
+z\rho _{2}\left( 1-\rho _{1}^{2}\right) }{\sqrt{1-\rho _{1}^{2}\rho _{2}^{2}%
}\sqrt{\rho _{1}^{2}+\rho _{2}^{2}-2\rho _{1}^{2}\rho _{2}^{2}}}.$
\end{proof}


\begin{thebibliography}{99}
\bibitem{Bo} Bo\.{z}ejko, Marek; K\"{u}mmerer, Burkhard; Speicher, Roland.
\$q\$-Gaussian processes: non-commutative and classical aspects. \emph{Comm.
Math. Phys.} \textbf{185} (1997), no. 1, 129--154. MR1463036 (98h:81053)

\bibitem{ASK89} Richard Askey. Continuous q-Hermite polynomials when q%
\TEXTsymbol{>} 1. \emph{In q-series and partitions (Minneapolis, MN, 1988)},
vol. \textbf{18} of IMA Vol. Math. Appl., pages 151--158. Springer, New
York, 1989. MR90h:33019

\bibitem{bryc1} Bryc, W\l odzimierz. Stationary random fields with linear
regressions. \emph{Ann. Probab.} \textbf{29} (2001), no. 1, 504--519.
MR1825162 (2002d:60014)

\bibitem{bms} Bryc, W\l odzimierz; Matysiak, Wojciech; Szab\l owski, Pawe\l\ %
J. Probabilistic aspects of Al-Salam--Chihara polynomials. \emph{Proc. Amer.
Math. Soc.} \textbf{133} (2005), no. 4, 1127--1134 (electronic). MR2117214
(2005m:33033)

\bibitem{AW85} Askey, Richard; Wilson, James. Some basic hypergeometric
orthogonal polynomials that generalize Jacobi polynomials. \emph{Mem. Amer.
Math. Soc.} \textbf{54} (1985), no. 319, iv+55 pp. MR0783216 (87a:05023)

\bibitem{Andrews1999} Andrews, George E.; Askey, Richard; Roy, Ranjan.
Special functions. Encyclopedia of Mathematics and its Applications, 71. 
\emph{Cambridge University Press,} Cambridge, 1999. xvi+664 pp. ISBN:
0-521-62321-9; 0-521-78988-5 MR1688958 (2000g:33001)

\bibitem{Corteel10} Corteel, Sylvie; Williams, Lauren K. Staircase tableaux,
the asymmetric exclusion process, and Askey--Wilson polynomials. Proc. Natl.
Acad. Sci. USA 107 (2010), no. 15, 6726--6730. MR2630104

\bibitem{Floreanini97} Floreanini, Roberto; LeTourneux, Jean; Vinet, Luc.
More on the \$q\$-oscillator algebra and \$q\$-orthogonal polynomials. \emph{%
J. Phys. A} 28 (1995), no. 10, L287--L293. MR1343867 (96e:33043)

\bibitem{Carlitz72} Carlitz, Leonard. Generating functions for certain
\$Q\$-orthogonal polynomials. \emph{Collect. Math.} \textbf{23} (1972),
91--104. MR0316773 (47 \#5321)

\bibitem{IA} Ismail, Mourad E. H. Classical and quantum orthogonal
polynomials in one variable. With two chapters by Walter Van Assche. With a
foreword by Richard A. Askey. Encyclopedia of Mathematics and its
Applications, 98. \emph{Cambridge University Press,} Cambridge, 2005.
xviii+706 pp. ISBN: 978-0-521-78201-2; 0-521-78201-5 MR2191786 (2007f:33001)

\bibitem{IS88} Ismail, Mourad E. H.; Stanton, Dennis. On the Askey--Wilson
and Rogers polynomials. \emph{Canad. J. Math.} \textbf{40} (1988), no. 5,
1025--1045. MR0973507 (89m:33003)

\bibitem{ISV87} Ismail, Mourad E. H.; Stanton, Dennis; Viennot, G\'{e}rard.
The combinatorics of \$q\$-Hermite polynomials and the Askey--Wilson
integral. \emph{European J. Combin.} \textbf{8} (1987), no. 4, 379--392.
MR0930175 (89h:33015)

\bibitem{AI84} Askey, Richard; Ismail, Mourad. Recurrence relations,
continued fractions, and orthogonal polynomials. \emph{Mem. Amer. Math. Soc.}
\textbf{49} (1984), no. 300, iv+108 pp. MR0743545 (85g:33008)

\bibitem{IsMas94} Ismail, Murad E. H.; Masson, David, R. \$q\$-Hermite
polynomials, biorthogonal rational functions, and \$q\$-beta integrals. 
\emph{Trans. Amer. Math. Soc.} \textbf{346} (1994), no. 1, 63--116.
MR1264148 (96a:33022)

\bibitem{Koek} Koekoek R. , Swarttouw R. F. (1999) The Askey-scheme of
hypergeometric orthogonal polynomials and its q-analogue, ArXiv:math/9602214

\bibitem{IRS99} Ismail, Mourad E. H.; Rahman, Mizan; Stanton, Dennis.
Quadratic \$q\$-exponentials and connection coefficient problems. \emph{%
Proc. Amer. Math. Soc.} \textbf{127} (1999), no. 10, 2931--2941. MR1621949
(2000a:33027)

\bibitem{matszab} Matysiak, Wojciech; Szab\l owski, Pawe\l\ J. A few remarks
on Bryc's paper on random fields with linear regressions. \emph{Ann. Probab.}
\textbf{30} (2002), no. 3, 1486--1491. MR1920274 (2003e:60111)

\bibitem{bryc05} Bryc, W\l odzimierz; Weso\l owski, Jacek. Conditional
moments of \$q\$-Meixner processes.\emph{\ Probab. Theory Related Fields }%
\textbf{131} (2005), no. 3, 415--441. MR2123251 (2005k:60233)

\bibitem{BryWes10} W\l odek Bryc and Jacek Weso\l owski, Askey--Wilson
polynomials, quadratic harnesses and martingales,\emph{\ Annals of
Probability}, \textbf{38}(3), (2010), 1221--1262

\bibitem{BryWe} Bryc, W\l odzimierz; Weso\l owski, Jacek. Bi-Poisson
process. \emph{Infin. Dimens. Anal. Quantum Probab. Relat.} \emph{Top.} 
\textbf{10} (2007), no. 2, 277--291. MR2337523 (2008d:60097)

\bibitem{BryBo} Bo\.{z}ejko, Marek; Bryc, W\l odzimierz. On a class of free L%
\'{e}vy laws related to a regression problem. \emph{J. Funct. Anal.} \textbf{%
236} (2006), no. 1, 59--77. MR2227129 (2007a:46071)

\bibitem{BryMaWe} Bryc, W\l odzimierz; Matysiak, Wojciech; Weso\l owski,
Jacek. The bi-Poisson process: a quadratic harness. \emph{Ann. Probab.} 
\textbf{36} (2008), no. 2, 623--646. MR2393992 (2009d:60103)

\bibitem{ALIs88} Al-Salam, Waleed. A.; Ismail, Mourad E. H. \$q\$-beta
integrals and the \$q\$-Hermite polynomials. \emph{Pacific J. Math.} \textbf{%
135} (1988), no. 2, 209--221. MR0968609 (90c:33001)

\bibitem{Szab} Szab\l owski, Pawe\l\ J. Probabilistic implications of
symmetries of \$q\$-Hermite and Al-Salam--Chihara polynomials. \emph{Infin.
Dimens. Anal. Quantum Probab. Relat. Top.} \textbf{11} (2008), no. 4,
513--522. MR2483794 (2010g:60125)

\bibitem{szab2} Szab\l owski, Pawe\l\ J. (2009) $q-$Gaussian Distributions:
Simplifications and Simulations, \emph{Journal of Probability and Statistics}%
, 2009 (article ID 752430)

\bibitem{Szab3} Szab\l owski, P. J. (2010) $q-$Wiener, $(\alpha ,q)-$%
Ornstein-Uhlenbeck processes. A generalization of known processes,
arXiv:math/0507303, submitted

\bibitem{Szab4} Szablowski, P.,J. (2010), Expansions of one density via
polynomials orthogonal with respect to the other.
http://arxiv.org/abs/1011.1492, submitted

\bibitem{Szab5} Szablowski, P.,J. (2010), Multidimensional $q$-Normal and
related distributions - Markov case\emph{, } Electronic J. of Probability, 
\textbf{15}(2010), paper no 40, pp. 1296-1318

\bibitem{Voi} Voiculescu, Dan. (2000)Lectures on free probability theory. 
\emph{Lectures on probability theory and statistics (Saint-Flour, 1998),
279--349, Lecture Notes in Math., 1738, Springer, Berlin, 2000}. MR1775641
(2001g:46121)

\bibitem{Nicu} Nica, Alexandru; Speicher, Roland. (2006) Lectures on the
combinatorics of free probability. \emph{London Mathematical Society Lecture
Note Series}, 335. \emph{Cambridge University Press, Cambridge, 2006}.
xvi+417 pp. ISBN: 978-0-521-85852-6; 0-521-85852-6 MR2266879

\bibitem{Voi2} Voiculescu, Dan. V.; Dykema, K. J.; Nica, Alexandru. Free
random variables. A noncommutative probability approach to free products
with applications to random matrices, operator algebras and harmonic
analysis on free groups. \emph{CRM Monograph Series, 1. American
Mathematical Society, Providence,} RI, 1992. vi+70 pp. ISBN: 0-8218-6999-X
MR1217253
\end{thebibliography}
\end{document}